\documentclass[12pt]{amsart}
\usepackage[linkcolor=blue, citecolor = blue]{hyperref}
\hypersetup{colorlinks=true}

\usepackage{amsthm,amsfonts,amsmath,amssymb,amscd} 

\usepackage{verbatim}
\usepackage{float}
\usepackage{wrapfig}
\usepackage{mathtools}
\usepackage[color=orange!20, linecolor=orange]{todonotes}

\usepackage{caption, subcaption}
\usepackage{enumitem}

\usepackage{ytableau}

\usepackage{tikz}
\usepackage{xcolor}  % For color definitions
\usetikzlibrary{decorations.markings, patterns,patterns.meta}
\tikzstyle{vertex}=[circle, draw, inner sep=0pt, minimum size=4pt]

\usetikzlibrary{arrows,automata}
\usepackage{tikz, tikz-3dplot}
\usepackage{tkz-graph}
\usetikzlibrary[positioning,patterns]
\usepackage{tikz-3dplot}
\usepackage{wasysym}

\newtheorem{theorem}{Theorem}[section]
\newtheorem{proposition}[theorem]{Proposition}
\newtheorem{lemma}[theorem]{Lemma}
\newtheorem{corollary}[theorem]{Corollary}
 
\theoremstyle{definition}
\newtheorem{definition}[theorem]{Definition}
\newtheorem{example}[theorem]{Example}

\theoremstyle{remark}
\newtheorem{remark}{Remark}

\newcommand{\sgn}{\mathrm{sgn}}

\renewcommand*{\det}{\qopname\relax o{det}}

\newcommand{\cor}{\mathrm{cor}}
\newcommand{\Cor}{\mathrm{Cor}}
\newcommand{\sh}{\mathrm{sh}}

\newcommand{\showAllXZ}[4]{ % 1-xl,2-xr,3-z,4-style
    \foreach \px in {#1, ..., #2}{
        \foreach \pz in {#3, ..., -1} {
            \draw[#4] (\px,0,\pz) -- (\px, 0, \pz + 1);
            \draw[#4] (\px,0,\pz) -- (\px - 1, 0, \pz + 1);
        }
    }
}

\newcommand{\drawrectZ}[7][cycle]{ % [1-cycle],2-style,3-z,4-xl,5-xr,6-yl,7-yr,
    \draw [#2] (#4, #6, #3) -- (#4, #7, #3) -- (#5, #7, #3) -- (#5, #6, #3) -- #1;
}
\newcommand{\drawplaneZ}[5]{ % 1-style,2-xl,3-xr,4-y,5-z
    \foreach \i in { 0,..., #4 } {
        \draw [-, #1] (#2, \i, #5) -- (#3, \i, #5);
    }
    \foreach \i in { #2,..., #3 } {
        \draw [-, #1] (\i, 0, #5) -- (\i, #4, #5);
    }
}
\newcommand{\drawplaneX}[5]{ % 1-style,2-x,3,4-yl,yr,5-z
    \foreach \i in { 0,..., #5 } {
        \draw [#1] (#2, 0, \i) -- (#2, 4, \i);
    }
    \foreach \i in { #3,..., #4 } {
        \draw [#1] (#2, \i, 0) -- (#2, \i, #5);
    }
}

\newcommand{\drawTop}[2]{%
  \coordinate (A) at ($#1$);
  \coordinate (B) at ($(A)+(-1,0,0)$);
  \coordinate (C) at ($(B)+(0,0,+1)$);
  \coordinate (D) at ($(C)+(+1,0,0)$);
  \draw[#2] (A) -- (B) -- (C) -- (D) -- cycle;
}
\newcommand{\drawLeft}[2]{%
  \coordinate (A) at ($#1$);
  \coordinate (B) at ($(A)+(-1,0,0)$);
  \coordinate (C) at ($(B)+(0,-1,0)$);
  \coordinate (D) at ($(C)+(+1,0,0)$);
  \draw[#2] (A) -- (B) -- (C) -- (D) -- cycle;
}
\newcommand{\drawRight}[2]{%
  \coordinate (A) at ($#1$);
  \coordinate (B) at ($(A)+(0,0,+1)$);
  \coordinate (C) at ($(B)+(0,-1,0)$);
  \coordinate (D) at ($(C)+(0,0,-1)$);
  \draw[#2] (A) -- (B) -- (C) -- (D) -- cycle;
}

\newcommand{\picturepp}{
  \begin{tikzpicture}[scale=0.75,x=1cm,y=1cm,z=0.3cm,
    rotate around x=-10,
    rotate around y=30,%-90
    rotate around z=-5 % 90
    ]
    \draw[->,dashed] (xyz cs:x=3) -- (xyz cs:x=4) node[below] {$k \le c$};
    \draw[->,dashed] (xyz cs:y=4) -- (xyz cs:y=4.5) node[above] {$i \le a$};
    \draw[->,dashed] (xyz cs:z=-3) -- (xyz cs:z=-4) node[below] {$j \le b$};
    \drawTop{(1,4,-1)}{-, black, thick, fill=cyan};
    \drawTop{(2,3,-1)}{-, black, thick, fill=cyan};
    \drawTop{(3,3,-1)}{-, black, thick, fill=cyan};
    
    \drawTop{(1,3,-2)}{-, black, thick, fill=cyan};
    \drawTop{(2,3,-2)}{-, black, thick, fill=cyan};
    \drawTop{(3,1,-2)}{-, black, thick, fill=cyan};

    \drawTop{(1,2,-3)}{-, black, thick, fill=cyan};
    \drawTop{(2,0,-3)}{-, black, thick}; % empty
    \drawTop{(3,0,-3)}{-, black, thick}; % empty
    \drawLeft{(2,4,0)}{-, black, thick}; % empty
    \drawLeft{(3,4,0)}{-, black, thick}; % empty
    \drawLeft{(1,4,-1)}{-, black, thick, fill=orange};
    
    \drawLeft{(1,3,-2)}{-, black, thick, fill=orange};
    \drawLeft{(2,3,-2)}{-, black, thick, fill=orange};
    \drawLeft{(3,3,-1)}{-, black, thick, fill=orange};

    \drawLeft{(1,2,-3)}{-, black, thick, fill=orange};
    \drawLeft{(2,2,-2)}{-, black, thick, fill=orange};
    \drawLeft{(3,2,-1)}{-, black, thick, fill=orange};

    \drawLeft{(1,1,-3)}{-, black, thick, fill=orange};
    \drawLeft{(2,1,-2)}{-, black, thick, fill=orange};
    \drawLeft{(3,1,-2)}{-, black, thick, fill=orange};
    
    \drawRight{(1,4,-1)}{-, black, thick}; \node at (1,3.5,-0.5) {\(1\)};
    \drawRight{(0,4,-2)}{-, black, thick}; \node at (0,3.5,-1.5) {\(0\)};
    \drawRight{(0,4,-3)}{-, black, thick}; \node at (0,3.5,-2.5) {\(0\)};

    \drawRight{(3,3,-1)}{-, black, thick}; \node at (3,2.5,-0.5) {$3$};
    \drawRight{(2,3,-2)}{-, black, thick}; \node at (2,2.5,-1.5) {$2$};
    \drawRight{(0,3,-3)}{-, black, thick}; \node at (0,2.5,-2.5) {$0$};

    \drawRight{(3,2,-1)}{-, black, thick}; \node at (3,1.5,-0.5) {$3$};
    \drawRight{(2,2,-2)}{-, black, thick}; \node at (2,1.5,-1.5) {$2$};
    \drawRight{(1,2,-3)}{-, black, thick}; \node at (1,1.5,-2.5) {$1$};

    \drawRight{(3,1,-1)}{-, black, thick}; \node at (3,0.5,-0.5) {$3$};
    \drawRight{(3,1,-2)}{-, black, thick}; \node at (3,0.5,-1.5) {$3$};
    \drawRight{(1,1,-3)}{-, black, thick}; \node at (1,0.5,-2.5) {$1$};
    \end{tikzpicture}
}
\newcommand{\picturecorner}{
  \begin{tikzpicture}[scale=0.75, x=1cm,y=1cm,z=0.3cm,
    rotate around x=-10,
    rotate around y=30,%-90
    rotate around z=-5 % 90
    ]
    \draw[->,dashed] (xyz cs:x=3) -- (xyz cs:x=4) node[below] {$k \le c$};
    \draw[->,dashed] (xyz cs:y=4) -- (xyz cs:y=4.5) node[above] {$i \le a$};
    \draw[->,dashed] (xyz cs:z=-3) -- (xyz cs:z=-4) node[below] {$j \le b$};
    \drawTop{(2,3,-1)}{-, gray, very thin, dashed};
    \drawTop{(3,0,-3)}{-, gray, very thin, dashed}; % empty
    \drawTop{(2,0,-3)}{-, gray, very thin, dashed}; % empty
    \drawLeft{(2,4,0)}{-, gray, very thin, dashed}; % empty
    \drawLeft{(3,4,0)}{-, gray, very thin, dashed}; % empty
    \drawLeft{(3,2,-1)}{-, gray, very thin, dashed};
    \drawLeft{(2,2,-2)}{-, gray, very thin, dashed};
    \drawLeft{(2,1,-2)}{-, gray, very thin, dashed};
    \drawLeft{(1,1,-3)}{-, gray, very thin, dashed};
    \drawRight{(1,4,-1)}{-, gray, very thin, dashed}; \node at (1,3.5,-0.5) {\(1\)};
    \drawRight{(0,4,-2)}{-, gray, very thin, dashed}; \node at (0,3.5,-1.5) {\(0\)};
    \drawRight{(0,4,-3)}{-, gray, very thin, dashed}; \node at (0,3.5,-2.5) {\(0\)};
    
    \drawRight{(3,3,-1)}{-, gray, very thin, dashed}; \node at (3,2.5,-0.5) {$3$};
    \drawRight{(2,3,-2)}{-, gray, very thin, dashed}; \node at (2,2.5,-1.5) {$2$};
    \drawRight{(0,3,-3)}{-, gray, very thin, dashed}; \node at (0,2.5,-2.5) {$0$};
    
    \drawRight{(3,2,-1)}{-, gray, very thin, dashed}; \node at (3,1.5,-0.5) {$3$};
    \drawRight{(2,2,-2)}{-, gray, very thin, dashed}; \node at (2,1.5,-1.5) {$2$};
    \drawRight{(1,2,-3)}{-, gray, very thin, dashed}; \node at (1,1.5,-2.5) {$1$};
    
    \drawRight{(3,1,-1)}{-, gray, very thin, dashed}; \node at (3,0.5,-0.5) {$3$};
    \drawRight{(3,1,-2)}{-, gray, very thin, dashed}; \node at (3,0.5,-1.5) {$3$};
    \drawRight{(1,1,-3)}{-, gray, very thin, dashed}; \node at (1,0.5,-2.5) {$1$};
    \drawTop{(1,4,-1)}{-, black, very thick, fill=cyan};
    \drawTop{(3,3,-1)}{-, black, very thick, fill=cyan};
    \drawTop{(1,3,-2)}{-, black, very thick, fill=cyan};
    \drawTop{(2,3,-2)}{-, black, very thick, fill=cyan};
    \drawTop{(3,1,-2)}{-, black, very thick, fill=cyan};
    \drawTop{(1,2,-3)}{-, black, very thick, fill=cyan};
    \drawLeft{(1,4,-1)}{-, black, very thick, fill=orange};
    \drawLeft{(1,3,-2)}{-, black, very thick, fill=orange};
    \drawLeft{(2,3,-2)}{-, black, very thick, fill=orange};
    \drawLeft{(3,3,-1)}{-, black, very thick, fill=orange};
    \drawLeft{(1,2,-3)}{-, black, very thick, fill=orange};
    \drawLeft{(3,1,-2)}{-, black, very thick, fill=orange};
\end{tikzpicture}
}

\newcommand{\pictureboxmainshift}{%
\begin{tikzpicture}[x=1cm,y=1cm,z=1cm,
    rotate around x=0,
    rotate around y=0,
    rotate around z=0
  ]
    \draw[->, dashed, very thick] (xyz cs:x=-3) -- (xyz cs:x=5) node[above] {$x(c)$};
    \draw[->, dashed, very thick] (xyz cs:y=2) -- (xyz cs:y=5) node[right] {$y(a)$};
    \draw[->, dashed, very thick] (xyz cs:z=0) -- (xyz cs:z=-3.5) node[left] {$z(b)$};
    \drawTop{(1,4,-1)}{-, darkgray, dashed, very thin};
    \drawTop{(1,4,-2)}{-, darkgray, dashed, very thin};
    \drawLeft{(1,4,-2)}{-, darkgray, dashed, very thin};
    \drawLeft{(1,3,-2)}{-, darkgray, dashed, very thin};
    \drawTop{(1,2,-3)}{-, darkgray, dashed, very thin};
    \drawLeft{(1,2,-3)}{-, darkgray, dashed, very thin};
    \drawLeft{(1,1,-3)}{-, darkgray, dashed, very thin};
    \drawTop{(2,4,-1)}{-, darkgray, dashed, very thin};
    \drawLeft{(2,4,-1)}{-, darkgray, dashed, very thin};
    \drawTop{(2,3,-2)}{-, darkgray, dashed, very thin};
    \drawLeft{(2,3,-2)}{-, darkgray, dashed, very thin};
    \drawTop{(2,2,-3)}{-, darkgray, dashed, very thin};
    \drawLeft{(2,2,-3)}{-, darkgray, dashed, very thin};
    \drawLeft{(2,1,-3)}{-, darkgray, dashed, very thin};
    \drawRight{(2,4,-1)}{-, darkgray, dashed, very thin};
    \drawTop{(3,3,-1)}{-, darkgray, dashed, very thin};
    \drawTop{(3,3,-2)}{-, darkgray, dashed, very thin};
    \drawLeft{(3,3,-2)}{-, darkgray, dashed, very thin};
    \drawLeft{(3,2,-2)}{-, darkgray, dashed, very thin};
    \drawTop{(3,1,-3)}{-, darkgray, dashed, very thin};
    \drawLeft{(3,1,-3)}{-, darkgray, dashed, very thin};
    \drawRight{(3,3,-1)}{-, darkgray, dashed, very thin};
    \drawTop{(4,2,-1)}{-, darkgray, dashed, very thin};
    \drawTop{(4,2,-2)}{-, darkgray, dashed, very thin};
    \drawLeft{(4,2,-2)}{-, darkgray, dashed, very thin};
    \drawTop{(4,1,-3)}{-, darkgray, dashed, very thin};
    \drawLeft{(4,1,-3)}{-, darkgray, dashed, very thin};
    \drawRight{(4,1,-3)}{-, darkgray, dashed, very thin};
    \drawRight{(4,1,-2)}{-, darkgray, dashed, very thin};
    \drawRight{(4,2,-2)}{-, darkgray, dashed, very thin};
    \drawRight{(4,1,-1)}{-, darkgray, dashed, very thin};
    \drawRight{(4,2,-1)}{-, darkgray, dashed, very thin};
    
    \filldraw[blue] (1,0,-3) circle (2pt) node[below] {$A_1$};
    \filldraw[black] (1,4,0) circle (2pt) node[above]{$B_1$};

    \draw [->, blue, very thick]
        (1, 0, -3) 
        -- (0, 0, -2)
        -- (-1, 0, -1) 
        -- (-1, 0, 0);
    \draw [->, black, very thick]
        (-1, 0, 0) -- (-1, 2, 0) 
        -- (0, 3, 0)
        -- (1, 4, 0);
    
    \filldraw[blue] (2,0,-3) circle (2pt) node[below] {$A_2$};
    \filldraw[black] (2,4,0) circle (2pt) node[above] {$B_2$};
    \draw [->, blue, very thick]
        (2, 0, -3) 
        -- (1, 0, -2)
        -- (0, 0, -1)
        -- (0, 0, 0);
    \draw [->, black, very thick]
        (0, 0, 0) 
        -- (0, 1, 0)
        -- (1, 2, 0)
        -- (1, 3, 0)
        -- (2, 4, 0);
    
    \filldraw[blue] (3,0,-3) circle (2pt) node[below] {$A_3$};
    \filldraw[black] (3,4,0) circle (2pt) node[above] {$B_3$};
    \draw [->, blue, very thick]
        (3, 0, -3)
        -- (2, 0, -2)
        -- (2, 0, -1)
        -- (1, 0, 0);
    \draw [->, black, very thick]
        (1, 0, 0)
        -- (1, 1, 0)
        -- (2, 2, 0)
        -- (3, 3, 0)
        -- (3, 4, 0);
    
    \filldraw[blue] (4,0,-3) circle (2pt) node[below] {$A_4$};
    \filldraw[black] (4,4,0) circle (2pt) node[above] {$B_4$};
    \draw [->, blue, very thick]
        (4, 0, -3)
        -- (3, 0, -2)
        -- (3, 0, -1)
        -- (2, 0, 0);
    \draw [->, black, very thick]
        (2, 0, 0)
        -- (3, 1, 0)
        -- (4, 2, 0)
        -- (4, 3, 0)
        -- (4, 4, 0);

    \filldraw[orange] (3.5, 0, -2.5) circle (2pt) node[left] {$q^3$};
    \filldraw[orange] (3.5, 0, 0) circle (2pt); 
    \filldraw[orange] (3.5, 1.5, 0) circle (2pt) node[left] {$q^1$};
    \draw[orange, very thick] (3.5, 0, -2.5) -- (3.5, 0, 0) -- (3.5, 1.5, 0);

    \filldraw[orange] (2.5, 0, -0.5) circle (2pt) node[left] {$q^1$};
    \filldraw[orange] (2.5, 0, 0) circle (2pt);
    \filldraw[orange] (2.5, 0.5, 0) circle (2pt) node[left] {$q^0$};
    \draw[orange, very thick] (2.5, 0, -0.5) -- (2.5, 0, 0) -- (2.5, 0.5, 0);
    
    \filldraw[blue] (-1,0,0) circle (1.5pt);
    \filldraw[blue] (0,0,0) circle (1.5pt);
    \filldraw[blue] (1,0,0) circle (1.5pt);
    \filldraw[blue] (2,0,0) circle (1.5pt);
  \end{tikzpicture}
}

\newcommand{\pictureprojcommon} {
    \draw[->, dashed, very thin, gray] (xyz cs:x=-4) -- (xyz cs:x=3) node[above] {$x$};
    \draw[->, dashed, very thin, gray] (xyz cs:y=0) -- (xyz cs:y=5) node[right] {$y$};
    \draw[->, dashed, very thin, gray] (xyz cs:z=0) -- (xyz cs:z=-6) node[above] {$z$};
    \drawplaneZ{dashed, gray, very thin}{-4}{0}{4}{0};
    \filldraw[black] (0, 0, -5) circle (2pt) node[below] {$A_i$};
    \filldraw[black] (0, 4, 0) circle (2pt) node[anchor=south east] {$B_i$};
    \filldraw[black] (-3, 0, 0) circle (2pt) node[anchor=south east] {$C_i$};
    
    \draw [-, blue, very thick] (0, 0, -5) -- (-1, 0, -4);
    \draw [-, blue, very thick] (-1, 0, -4) --  (-1, 0, -3);
    \draw [-, blue, very thick] (-1, 0, -3) --  (-2, 0, -2);
    \draw [-, blue, very thick] (-2, 0, -2) --  (-3, 0, -1);
    \draw [-, blue, very thick] (-3, 0, -1) --  (-3, 0, 0);
    \draw [-, black, thick] (-3, 0, 0) -- (-2, 1, 0);
    \draw [-, black, thick] (-2, 1, 0) --  (-2, 2, 0);
    \draw [-, black, thick] (-2, 2, 0) --  (-1, 3, 0);
    \draw [->, black, thick] (-1, 3, 0) --  (0, 4, 0);
}

\newcommand{\pictureprojA}{
\begin{tikzpicture}[->,x=1cm,y=1cm,z=1cm,rotate around y=0,rotate around x=0, rotate around z=0, scale=0.5]
    \pictureprojcommon
    \draw [-, orange, thin] (0, 0, -5) -- (0,0,0);
    \draw [-, orange, thin] (0, 1, -4) -- (0,1,0);
    \draw [-, orange, thin] (0, 2, -4) -- (0,2,0);
    \draw [-, orange, thin] (0, 3, -2) -- (0,3,0);
    
    \draw [-, orange, thin] (0, 4, 0) -- (0, 0, 0);
    \draw [-, orange, thin] (0, 4, -1) -- (0,0,-1);
    \draw [-, orange, thin] (0, 3, -2) -- (0,0,-2);
    \draw [-, orange, thin] (0, 3, -3) -- (0,0,-3);
    \draw [-, orange, thin] (0, 1, -4) -- (0,0,-4);
    \draw [-, orange, dashed, very thick] (0, 4, -2) -- (0,4,-5);
    \filldraw[orange] (0, 4, -5) circle (2pt);
    \draw [-, orange, dashed, very thick] (0, 4, -5) -- (0,1,-5);
    \draw[fill=white, draw=white] (0, 1.5, -1.5) -- (0, 1.5, -2.5)
      -- (0, 2.5, -2.5) -- (0, 2.5, -1.5) -- cycle;
    \node at (0,2,-2) {\(\lambda^{(i)}\)};
\end{tikzpicture}
}

\newcommand{\pictureprojB}{
\begin{tikzpicture}[->,x=1cm,y=1cm,z=0.5cm,rotate around y=-80,rotate around x=0, scale=0.7]
    \pictureprojcommon
    \drawplaneX{-, orange, dashed, very thin}{0}{0}{4}{-5};
    \draw [-, orange, very thick]  (0, 4, 0) 
      -- (0, 4, -1)
      -- (0, 4, -2)
      -- (0, 3, -2)
      -- (0, 3, -3) % C
      -- (0, 3, -4)
      -- (0, 1, -4)
      -- (0, 1, -5)
      -- (0, 0, -5)
      ;
    \draw[<-, blue, dashed, thick] (-1,0,-4) -- (0, 1, -5);
    \draw[<-, blue, dashed, thick] (-1,0,-3) -- (0, 1, -4);
    \draw[<-, blue, dashed, thick] (-2,0,-2) -- (0, 2, -4);
    \draw[<-, blue, dashed, thick] (-3,0,-1) -- (0, 3, -4);
    \draw[<-, black, dashed, thick] (-3,0,0) -- (0, 3, -3);
    \draw[<-, black, dashed, thick] (-2,1,0) -- (0, 3, -2);
    \draw[<-, black, dashed, thick] (-2,2,0) -- (0, 4, -2);
    \draw[<-, black, dashed, thick] (-2,2,0) -- (0, 4, -2);
    \draw[<-, black, dashed, thick] (-1,3,0) -- (0, 4, -1);
    \draw[<-, black, dashed, thick] (-1,3,0) -- (0, 4, -1);
    \draw [->, gray, dashed, very thin] (-3,0,-1) -- (0, 0, -1);
    \draw [->, gray, dashed, very thin] (-2,0,-2) -- (0, 0, -2);
    \draw [->, gray, dashed, very thin] (-1,0,-3) -- (0, 0, -3);
    \draw [->, gray, dashed, very thin] (-1,0,-4) -- (0, 0, -4);
    \draw [->, gray, dashed, very thin] (0,0,-5) -- (0, 0, 0);
\end{tikzpicture}
}

\newcommand{\pictureonecommon} {
    \draw[->, dashed, very thin, gray] (xyz cs:x=-4) -- (xyz cs:x=3) node[above] {$x$};
    \draw[->, dashed, very thin, gray] (xyz cs:y=0) -- (xyz cs:y=5) node[left] {$y$};
    \draw[->, dashed, very thin, gray] (xyz cs:z=0) -- (xyz cs:z=-6) node[above] {$z$};
    \drawplaneZ{dashed, gray, very thin}{-4}{0}{4}{0};
    \filldraw[black] (0, 0, -5) circle (2pt) node[below] {$A$};
    \filldraw[black] (0, 4, 0) circle (2pt) node[anchor=east] {$B$};
    \draw [->, gray, dashed, very thin] (-3,0,-1) -- (0, 0, -1);
    \draw [->, gray, dashed, very thin] (-2,0,-2) -- (0, 0, -2);
    \draw [->, gray, dashed, very thin] (-1,0,-3) -- (0, 0, -3);
    \draw [->, gray, dashed, very thin] (-1,0,-4) -- (0, 0, -4);
    \draw [->, gray, dashed, very thin] (0,0,-5) -- (0, 0, 0);
}

\newcommand{\pictureoneslideA}{
\begin{tikzpicture}[->,x=1cm,y=1cm,z=0.5cm,rotate around y=-80,rotate around x=0, scale=0.6]
    \pictureonecommon
    \drawplaneX{-, orange, dashed, very thin}{0}{0}{4}{-5};
    \draw [-, orange, very thick]  (0, 4, 0) 
      -- (0, 4, -1)
      -- (0, 4, -2)
      -- (0, 3, -2)
      -- (0, 3, -3) % C
      -- (0, 3, -4)
      -- (0, 1, -4)
      -- (0, 1, -5)
      -- (0, 0, -5)
      ;
    \draw [-, blue, very thick] (0, 0, -5) -- (-1, 0, -4);
    \draw [-, blue, very thick] (-1, 0, -4) --  (-1, 0, -3);
    \draw [-, blue, very thick] (-1, 0, -3) --  (-2, 0, -2);
    \draw [-, blue, very thick] (-2, 0, -2) --  (-3, 0, -1);
    \draw [-, blue, very thick] (-3, 0, -1) --  (-3, 0, 0);
    \draw [-, black, thick] (-3, 0, 0) -- (-2, 1, 0);
    \draw [-, black, thick] (-2, 1, 0) --  (-2, 2, 0);
    \draw [-, black, thick] (-2, 2, 0) --  (-1, 3, 0);
    \draw [-, black, thick] (-1, 3, 0) --  (0, 4, 0);
    \node[color=black] at (0,5,-3) {$\lambda=(6,4,4,2)$};
\end{tikzpicture}
}
\newcommand{\pictureoneslideB}{
\begin{tikzpicture}[->,x=1cm,y=1cm,z=0.5cm,rotate around y=-80,rotate around x=0, scale=0.6]
    \pictureonecommon
    \drawplaneX{-, orange, dashed, very thin}{0}{0}{4}{-5};
    \draw [-, orange, very thick]  
         (0, 4, 0) 
      -- (0, 4, -1);
    \draw [-, blue, very thick] (0, 0, -5) -- (-1, 0, -4);
    \draw [-, blue, very thick] (-1, 0, -4) --  (-1, 0, -3);
    \draw [-, blue, very thick] (-1, 0, -3) --  (-2, 0, -2);
    \draw [-, blue, very thick] (-2, 0, -2) --  (-3, 0, -1);
    \draw [-, blue, dashed] (-3, 0, -1) --  (-3, 0, 0);
    \draw [-, black, thick] (-3, 0, -1) -- (-2, 1, -1);
    \draw [-, black, thick] (-2, 1, -1) --  (-2, 2, -1);
    \draw [-, black, thick] (-2, 2, -1) --  (-1, 3, -1);
    \draw [-, black, thick] (-1, 3, -1) --  (0, 4, -1);
\end{tikzpicture}
}

\newcommand{\pictureoneslideC}{
\begin{tikzpicture}[->,x=1cm,y=1cm,z=0.5cm,rotate around y=-80,rotate around x=0, scale=0.6]
    \pictureonecommon
    \drawplaneX{-, orange, dashed, very thin}{0}{0}{4}{-5};
    \draw [-, orange, very thick]  
         (0, 4, 0) 
      -- (0, 4, -1)
      -- (0, 3, -2);
    \draw [-, blue, very thick] (0, 0, -5) -- (-1, 0, -4);
    \draw [-, blue, very thick] (-1, 0, -4) --  (-1, 0, -3);
    \draw [-, blue, very thick] (-1, 0, -3) --  (-2, 0, -2);
    \draw [-, blue, dashed] (-2, 0, -2) --  (-3, 0, -1);
    \draw [-, blue, dashed] (-3, 0, -1) --  (-3, 0, 0);
    \draw [-, black, thick] (-2, 0, -2) -- (-1, 1, -2);
    \draw [-, black, thick] (-1, 1, -2) --  (-1, 2, -2);
    \draw [-, black, thick] (-1, 2, -2) --  (0, 3, -2);
\end{tikzpicture}
}
\newcommand{\pictureoneslideD}{
\begin{tikzpicture}[->,x=1cm,y=1cm,z=0.5cm,rotate around y=-80,rotate around x=0, scale=0.6]
    \pictureonecommon
    \drawplaneX{-, orange, dashed, very thin}{0}{0}{4}{-5};
    \draw [-, orange, very thick]  
         (0, 4, 0) 
      -- (0, 4, -1)
      -- (0, 3, -2)
      -- (0, 2, -3);
    \draw [-, blue, very thick] (0, 0, -5) -- (-1, 0, -4);
    \draw [-, blue, very thick] (-1, 0, -4) --  (-1, 0, -3);
    \draw [-, blue, dashed] (-1, 0, -3) --  (-2, 0, -2);
    \draw [-, blue, dashed] (-2, 0, -2) --  (-3, 0, -1);
    \draw [-, blue, dashed] (-3, 0, -1) --  (-3, 0, 0);
    \draw [-, black, thick] (-1, 0, -3) -- (0, 1, -3);
    \draw [-, orange, thick] (0, 1, -3) --  (0, 2, -3);
\end{tikzpicture}
}
\newcommand{\pictureoneslideE}{
\begin{tikzpicture}[->,x=1cm,y=1cm,z=0.5cm,rotate around y=-80,rotate around x=0, scale=0.6]
    \pictureonecommon
    \drawplaneX{-, orange, dashed, very thin}{0}{0}{4}{-5};
    \draw [-, orange, very thick]  
         (0, 4, 0) 
      -- (0, 4, -1)
      -- (0, 3, -2)
      -- (0, 2, -3)
      -- (0, 1, -3)
      -- (0, 1, -4);
    \draw [-, blue, very thick] (0, 0, -5) -- (-1, 0, -4);
    \draw [-, blue, dashed] (-1, 0, -4) --  (-1, 0, -3);
    \draw [-, blue, dashed] (-1, 0, -3) --  (-2, 0, -2);
    \draw [-, blue, dashed] (-2, 0, -2) --  (-3, 0, -1);
    \draw [-, blue, dashed] (-3, 0, -1) --  (-3, 0, 0);
    \draw [-, black, thick] (-1, 0, -4) -- (0, 1, -4);
\end{tikzpicture}
}
\newcommand{\pictureoneslideF}{
\begin{tikzpicture}[->,x=1cm,y=1cm,z=0.5cm,rotate around y=-80,rotate around x=0, scale=0.6]
    \pictureonecommon
    \drawplaneX{-, orange, dashed, very thin}{0}{0}{4}{-5};
    \draw [-, orange, very thick]  
         (0, 4, 0) 
      -- (0, 4, -1)
      -- (0, 3, -2)
      -- (0, 2, -3)
      -- (0, 1, -3)
      -- (0, 1, -4)
      -- (0, 0, -5);
    \draw [-, blue, dashed] (0, 0, -5) -- (-1, 0, -4);
    \draw [-, blue, dashed] (-1, 0, -4) --  (-1, 0, -3);
    \draw [-, blue, dashed] (-1, 0, -3) --  (-2, 0, -2);
    \draw [-, blue, dashed] (-2, 0, -2) --  (-3, 0, -1);
    \draw [-, blue, dashed] (-3, 0, -1) --  (-3, 0, 0);
    \node[color=black] at (0,5,-3.5) {$\Phi(\lambda)=(5,3,3,2)$};
\end{tikzpicture}
}

\newcommand{\picturenpath}{
\begin{tikzpicture}[x=1cm,y=1cm,z=0.5cm,rotate around y=30,rotate around x=0, rotate around z=0, scale = 0.7]\label{graph1}
    \draw[dashed, ->] (xyz cs:x=0) -- (xyz cs:x=5.5) node[above] {$x$};
    \draw[dashed, ->, very thin, lightgray] (xyz cs:y=0) -- (xyz cs:y=4.3) node[right] {$y$};
    \draw[->, dashed] (xyz cs:z=0) -- (xyz cs:z=-4) node[above] {$z$};
    \drawplaneZ{lightgray}{-1}{4}{4}{0}
    \showAllXZ{0}{4}{-4}{ultra thin, dashed, lightgray}
    \node[below] at (1,0,-3){${A_1}$};
    \node[below] at (2,0,-3){${A_2}$};
    \node[below] at (3,0,-3){${A_3}$};
    \node[below] at (4,0,-3){${A_4}$};
    \node[above] at (1,4,0){${B_1}$};
    \node[above] at (2,4,0){${B_2}$};
    \node[above] at (3,4,0){${B_3}$};
    \node[above] at (4,4,0){${B_4}$};
    \draw [blue, very thick] (1,0,-3)
        -- (0,0,-2) node[left, pos=0.4]{$z_3$}
        -- (-1,0,-1) node[left, pos=0.7]{$z_2$}
        -- (-1,0,0);
    \draw [black, very thick] (-1,0,0)
        -- (-1,1,0)
        -- (-1,2,0)
        -- (0,3,0) node[left, pos=0.7]{$x_3$}
        -- (1,4,0) node[left, pos=0.7]{$x_4$};
    \draw [blue, very thick] (2,0,-3)
        -- (1,0,-2) node[left, pos=0.7]{$z_3$}
        -- (1,0,-1) 
        -- (0,0,0) node[left, pos=0.7]{$z_1$};
    \draw [black, very thick] (0,0,0)
        -- (1,1,0) node[left, pos=0.7]{$x_1$}
        -- (1,2,0) 
        -- (2,3,0) node[left, pos=0.7]{$x_3$}
        -- (2,4,0);
    \draw [blue, very thick] (3,0,-3)
        -- (3,0,-2)
        -- (2,0,-1) node[left, pos=0.7]{$z_2$}
        -- (2,0,0);
    \draw [black, very thick] (2,0,0)
        -- (2,1,0)
        -- (2,2,0)
        -- (3,3,0) node[left, pos=0.7]{$x_3$}
        -- (3,4,0);
    \draw [blue, very thick] (4,0,-3) -- (4,0,0) node[above, midway]{$1$};
    \draw [black, very thick] (4,0,0) -- (4,4,0) node[left, midway]{$1$};
    \filldraw[black] (-1,0,0) circle (2pt);
    \filldraw[black] (0,0,0) circle (2pt);
    \filldraw[black] (2,0,0) circle (2pt);
    \filldraw[black] (4,0,0) circle (2pt);
  \end{tikzpicture}
}
\newcommand{\picturenpathy}{
  \begin{tikzpicture}[x=1cm,y=1cm,z=0.5cm,rotate around x=-63,rotate around z=-24, scale = 0.4]\label{graph2}    
    \draw[dashed, ->] (xyz cs:x=-1.5) -- (xyz cs:x=6) node[above] {$x$};
    \draw[->, dashed] (xyz cs:z=0) -- (xyz cs:z=-5) node[right] {$z$};
    \draw[->, dashed] (0,0,0) -- (0,0,5) node[left] {$y$};
    \showAllXZ{0}{5}{-4}{ultra thin, dashed, lightgray}
    \showAllXZ{0}{5}{3}{ultra thin, dashed, lightgray}
    \draw[<->, gray, thick] (4,0,-4) -- (0,0,-4) node[midway, below] {$c$};
    \draw[<->, gray, thick] (5,0,4) -- (5,0,0.1) node[midway, right] {$a$};
    \draw[<->, gray, thick] (5,0,-3) -- (5,0,-0.1) node[midway, right] {$b$};
    \draw[-, black, thick, dashed] (0, 0, -3)
        -- (4, 0, -3)
        -- (4, 0, 0);
    \filldraw[blue] (1,0,-3) circle (2pt) node[below, black] {${A_1}$};
    \filldraw[blue] (2,0,-3) circle (2pt) node[below, black] {${A_2}$};
    \filldraw[blue] (3,0,-3) circle (2pt) node[below, black] {${A_3}$};
    \filldraw[blue] (4,0,-3) circle (2pt) node[below, black] {${A_4}$};
    \filldraw[black] (1,4,4) circle (2pt) node[above] {${B_1}$};
    \filldraw[black] (2,4,4) circle (2pt) node[above] {${B_2}$};
    \filldraw[black] (3,4,4) circle (2pt) node[above] {${B_3}$};
    \filldraw[black] (4,4,4) circle (2pt) node[above] {${B_4}$};
    
    \draw [blue, very thick] (1,0,-3)
    -- (0,0,-2)
    -- (-1,0,-1)
    -- (-1,0,0);
    \draw [black, very thick] (-1,0,0)
    -- (-1,1,1)
    -- (-1,2,2)
    -- (0,3,3)
    -- (1,4,4);
    \draw [blue, very thick] (2,0,-3)
    -- (1,0,-2)
    -- (1,0,-1) 
    -- (0,0,0);
    \draw [black, very thick] (0,0,0)
    -- (1,1,1)
    -- (1,2,2) 
    -- (2,3,3)
    -- (2,4,4);
    \draw [blue, very thick] (3,0,-3)
    -- (3,0,-2)
    -- (2,0,-1)
    -- (2,0,0);
    \draw [black, very thick] (2,0,0)
    -- (2,1,1)
    -- (2,2,2)
    -- (3,3,3)
    -- (3,4,4);
    \draw [blue, very thick] (4,0,-3) -- (4,0,0);
    \draw [black, very thick] (4,0,0) -- (4,4,4);
    \filldraw[black] (-1,0,0) circle (2pt);
    \filldraw[black] (0,0,0) circle (2pt);
    \filldraw[black] (2,0,0) circle (2pt);
    \filldraw[black] (4,0,0) circle (2pt);
  \end{tikzpicture}
}

\newcommand{\showXZweak}[3]{ % 1-x, 2-z
    \foreach \pz in {#2, ..., -1} {
        \draw[->, #3] (#1,0,\pz) -- (#1, 0, \pz + 1);
        \draw[->, #3] (#1,0,\pz) -- (#1 - 1, 0, \pz + 1);
    }
}
\newcommand{\picturegraph}{
\begin{tikzpicture}[x=1cm,y=1cm,z=0.5cm,rotate around y=-45,rotate around x=-0, scale = 0.8]\label{graph3}

    \showXZweak{3}{-3}{very thin, lightgray, dashed}
    \showXZweak{2}{-3}{very thin, lightgray, dashed}
    \showXZweak{1}{-3}{very thin, lightgray, dashed}
   
    \drawplaneZ{dashed, lightgray, thin}{0}{3}{3}{0};
    \draw [->, black, very thick] (2,2,0) -- (2, 3, 0) node[left, pos=0.5]{$1$};
    \draw [->, black, very thick] (2,2,0) -- (3, 3, 0) node[right, pos=0.5]{$x_i$};
    
    \draw[->, dashed] (xyz cs:x=0) -- (xyz cs:x=4) node[above] {$x$};
    \draw[->, dashed] (xyz cs:y=0) -- (xyz cs:y=4) node[right] {$y$};
    \draw[->, dashed] (xyz cs:z=0) -- (xyz cs:z=-3.5) node[above] {$z$};

    \draw [->, blue, very thick] (3,0,-3) -- (3, 0, -2) node[right, pos=0.7]{$~1$};
    \draw [->, blue, very thick] (3,0,-3) -- (2, 0, -2)node[below, pos=0.5]{$z_i$};

    \end{tikzpicture}
}

\newcommand{\pictureslideA}{
\begin{tikzpicture}[x=1.2cm,y=1cm,z=0.5cm,rotate around y=33,rotate around x=0, scale=1.0]
    \draw[->, dashed, lightgray] (xyz cs:x=-2) -- (xyz cs:x=5.3) node[above] {$x$};
    \draw[->, dashed, lightgray] (xyz cs:y=0) -- (xyz cs:y=4.3) node[right] {$y$};
    \draw[->, dashed, lightgray] (xyz cs:z=0) -- (xyz cs:z=-3) node[above] {$z$};
    \showAllXZ{0}{4}{-3}{ultra thin, dashed, lightgray};
    \drawrectZ{gray, thin}{-1}{-1}{5}{0}{4}
    \drawrectZ{gray, thin}{-2}{-1}{5}{0}{4}
    \drawrectZ{orange, very thick}{0}{-1}{5}{0}{4};
    \filldraw[black] (2,4,0) circle (1pt) node[above] {$B_1$};
    \filldraw[black] (3,4,0) circle (1pt) node[above] {$B_2$};
    \filldraw[black] (4,4,0) circle (1pt) node[above] {$B_3$};
    \filldraw[black] (2,0,-2) circle (1pt) node[below] {$A_1$};
    \filldraw[black] (3,0,-2) circle (1pt) node[below] {$A_2$};
    \filldraw[black] (4,0,-2) circle (1pt) node[below] {$A_3$};
    \draw [->, blue, very thick] (2, 0, -2) 
        -- (1, 0, -1) node[pos=0.6, left]{$z_2$}
        -- (0, 0, 0) node[pos=0.6, left]{$z_1$};
    \draw [->, black, very thick] (0, 0, 0) 
        -- (0, 1, 0)
        -- (0, 2, 0) 
        -- (1, 3, 0) node[pos=0.6, left, black]{$x_3$}
        -- (2, 4, 0) node[pos=0.6, left, black]{$x_4$};
    \draw [->, blue, very thick] (3, 0, -2) 
        -- (2, 0, -1) node[pos=0.6, left]{$z_2$} 
        -- (2, 0, 0);
    \draw [->, black, very thick] (2, 0, 0)
        -- (2, 2, 0)
        -- (3, 3, 0) node[pos=0.6, left, black]{$x_3$}
        -- (3, 4, 0);
    \draw [->, blue, very thick] (4, 0, -2) 
        -- (4, 0, -1)
        -- (3, 0, 0) node[pos=0.6, left]{$z_1$};
    \draw [->, black, very thick] (3, 0, 0) 
        -- (3, 1, 0)
        -- (4, 2, 0) node[pos=0.6, left, black]{$x_2$} 
        -- (4, 4, 0);
\end{tikzpicture}
}

\newcommand{\pictureslideB}{
\begin{tikzpicture}[x=1.2cm,y=1cm,z=0.5cm,rotate around y=33,rotate around x=0]
    \draw[->, dashed, lightgray] (xyz cs:x=-2) -- (xyz cs:x=5.3) node[above] {$x$};
    \draw[->, dashed, lightgray] (xyz cs:y=0) -- (xyz cs:y=4.3) node[right] {$y$};
    \draw[->, dashed, lightgray] (xyz cs:z=0) -- (xyz cs:z=-3) node[above] {$z$};
    \showAllXZ{0}{4}{-3}{ultra thin, dashed, lightgray};
    \drawrectZ{orange, very thick}{-1}{-1}{5}{0}{4};
    \drawrectZ{gray, thin}{-2}{-1}{5}{0}{4};
    \drawrectZ{gray, thin}{0}{-1}{5}{0}{4};
    \filldraw[black] (2,4,0) circle (1pt) node[above] {$B_1$};
    \filldraw[black] (3,4,0) circle (1pt) node[above] {$B_2$};
    \filldraw[black] (4,4,0) circle (1pt) node[above] {$B_3$};
    \filldraw[black] (2,0,-2) circle (1pt) node[below] {$A_1$};
    \filldraw[black] (3,0,-2) circle (1pt) node[below] {$A_2$};
    \filldraw[black] (4,0,-2) circle (1pt) node[below] {$A_3$};
    \draw [->, blue, very thick] (2, 0, -2) 
        -- (1, 0, -1) node[pos=0.6, left]{$z_2$};
    \draw [->, black, very thick] (1, 0, -1) 
        -- (1, 1, -1)
        -- (1, 2, -1) 
        -- (2, 3, -1) node[pos=0.6, left, black]{$x_3$};
        \draw [->, orange, very thick] (2, 3, -1) 
        -- (2, 4, 0) node[pos=0.8, left, orange]{$x_4 z_1$};
    \draw [->, blue, very thick] (3, 0, -2) 
        -- (2, 0, -1) node[pos=0.6, left]{$z_2$};
    \draw [->, black, very thick] (2, 0, -1)
        -- (2, 2, -1)
        -- (3, 3, -1) node[pos=0.6, left, black]{$x_3$};
    \draw [->, orange, very thick] (3, 3, -1) -- (3, 3, 0) -- (3,4,0);
    \draw [->, blue, very thick] (4, 0, -2) 
        -- (4, 0, -1);
    \draw [->, black, very thick] (4, 0, -1) 
        -- (4, 1, -1);
    \draw [->, orange, very thick] (4, 1, -1) 
        -- (4, 2, 0) node[pos=0.6, left, orange]{$x_2 z_1$}
        -- (4, 4, 0);
    \end{tikzpicture}
}

\newcommand{\pictureslideC}{
\begin{tikzpicture}[x=1.2cm,y=1cm,z=0.5cm,rotate around y=33,rotate around x=0]
    \draw[->, dashed, lightgray] (xyz cs:x=-2) -- (xyz cs:x=5.3) node[above] {$x$};
    \draw[->, dashed, lightgray] (xyz cs:y=0) -- (xyz cs:y=4.3) node[right] {$y$};
    \draw[->, dashed, lightgray] (xyz cs:z=0) -- (xyz cs:z=-3) node[above] {$z$};
    \showAllXZ{0}{4}{-3}{ultra thin, dashed, lightgray};
    \drawrectZ{orange, very thick}{-2}{-1}{5}{0}{4}
    \drawrectZ{gray, thin}{-1}{-1}{5}{0}{4}
    \drawrectZ{gray, thin}{0}{-1}{5}{0}{4};
    \filldraw[black] (2,4,0) circle (1pt) node[above] {$B_1$};
    \filldraw[black] (3,4,0) circle (1pt) node[above] {$B_2$};
    \filldraw[black] (4,4,0) circle (1pt) node[above] {$B_3$};
    \filldraw[black] (2,0,-2) circle (1pt) node[below] {$A_1$};
    \filldraw[black] (3,0,-2) circle (1pt) node[below] {$A_2$};
    \filldraw[black] (4,0,-2) circle (1pt) node[below] {$A_3$};
    \draw [->, orange, very thick] (2, 0, -2) -- (2, 1, -2);
    \draw [->, orange, very thick] (2, 1, -2) -- (2, 2, -2);
    \draw [->, orange, very thick] (2, 2, -2) -- (2, 3, -1) node[pos=0.6, left, orange]{$x_4 z_2$};
    \draw [->, orange, very thick] (2, 3, -1) -- (2, 4, 0) node[pos=0.6, left, orange]{$x_4 z_1$};
    \draw [->, orange, very thick] (3, 0, -2) -- (3, 1, -2);
    \draw [->, orange, very thick] (3, 1, -2) -- (3, 2, -2);
    \draw [->, orange, very thick] (3, 2, -2) -- (3, 3, -1) node[pos=0.3, right, orange]{$x_3 z_2$};
    \draw [->, orange, very thick] (3, 3, -1) -- (3, 3, 0);
    \draw [->, orange, very thick] (3, 3, 0) -- (3, 4, 0);
    \draw [->, orange, very thick] (4, 0, -2) -- (4, 0, -1);
    \draw [->, orange, very thick] (4, 0, -1) -- (4, 1, -1);
    \draw [->, orange, very thick] (4, 1, -1) -- (4, 2, 0) node[pos=0.6, left, orange]{$x_2 z_1$};
    \draw [->, orange, very thick] (4, 2, 0) -- (4, 3, 0);
    \draw [->, orange, very thick] (4, 3, 0) -- (4, 4, 0);
    \end{tikzpicture}
}

\newcommand{\pictureslideD}{
  \begin{tikzpicture}[scale=0.5, x=1cm,y=1cm,z=0.3cm,
    rotate around x=-10,
    rotate around y=30,%-90
    rotate around z=-5 % 90
    ]
    \tiny
    \draw[->, gray, dashed] (xyz cs:x=3) -- (xyz cs:x=4) node[below] {$x$};
    \draw[->, gray, dashed] (xyz cs:y=4) -- (xyz cs:y=4.5) node[right] {$y$};
    \draw[->, gray, dashed] (xyz cs:z=-2) -- (xyz cs:z=-3) node[below] {$z$};
    \drawTop{(2,3,-1)}{-, gray, very thin, dashed};
    \drawLeft{(2,4,0)}{-, gray, very thin, dashed}; % empty
    \drawLeft{(3,4,0)}{-, gray, very thin, dashed}; % empty
    \drawLeft{(3,3,0)}{-, gray, very thin, dashed};
    \drawLeft{(3,2,-1)}{-, gray, very thin, dashed};
    \drawLeft{(2,2,-2)}{-, gray, very thin, dashed};
    \drawLeft{(2,1,-2)}{-, gray, very thin, dashed};
    \drawRight{(1,4,-1)}{-, gray, very thin, dashed}; \node at (1,3.5,-0.5) {\(1\)};
    \drawRight{(0,4,-2)}{-, gray, very thin, dashed}; \node at (0,3.5,-1.5) {\(0\)};
    \drawRight{(2,3,-1)}{-, gray, very thin, dashed}; \node at (2,2.5,-0.5) {$2$};
    \drawRight{(2,3,-2)}{-, gray, very thin, dashed}; \node at (2,2.5,-1.5) {$2$};
    \drawRight{(3,2,-1)}{-, gray, very thin, dashed}; \node at (3,1.5,-0.5) {$3$};
    \drawRight{(2,2,-2)}{-, gray, very thin, dashed}; \node at (2,1.5,-1.5) {$2$};
    \drawRight{(3,1,-1)}{-, gray, very thin, dashed}; \node at (3,0.5,-0.5) {$3$};
    \drawRight{(2,1,-2)}{-, gray, very thin, dashed}; \node at (2,0.5,-1.5) {$2$};
    \drawLeft{(1,2,-2)}{-, gray, very thin, dashed};
    \drawLeft{(1,1,-2)}{-, gray, very thin, dashed};
    \drawTop{(1,4,-1)}{-, black, thick, fill=cyan};
    \drawTop{(3,2,-1)}{-, black, thick, fill=cyan};
    \drawTop{(1,3,-2)}{-, black, thick, fill=cyan};
    \drawTop{(2,3,-2)}{-, black, thick, fill=cyan};
    \drawTop{(3,0,-2)}{gray, very thin, dashed};
    \drawLeft{(1,4,-1)}{-, black, thick, fill=orange};
    \drawLeft{(1,3,-2)}{-, black, thick, fill=orange};
    \drawLeft{(2,3,-2)}{-, black, thick, fill=orange};
    \drawLeft{(3,2,-1)}{-, black, thick, fill=orange};
    \drawLeft{(3,1,-1)}{gray, very thin, dashed};
\end{tikzpicture}
}

\title[On equidistribution theorem for plane partitions]{
On equidistribution theorem for plane partitions
} 

\author[Alimzhan Amanov \and Damir Yeliussizov]{Alimzhan Amanov \and Damir Yeliussizov}
\address{Kazakh-British Technical University, Almaty, Kazakhstan}
\email{\href{mailto:alimzhan.amanov@gmail.com}{alimzhan.amanov@gmail.com}}

\email{\href{mailto:yeldamir@gmail.com}{yeldamir@gmail.com}}

\begin{document}
\begin{abstract}
  We prove equidistribution of two pairs of statistics on boxed plane partitions: (volume, trace) and (corner-hook volume, number of corners). The proof relies on different 3d visualizations of the corresponding non-intersecting path systems. In particular, we obtain a new visual proof for a volume generating function of plane partitions. We also introduce a new statistic called the cohook area on ordinary partitions, and prove that it is equidistributed with the area of partitions. 
\end{abstract}

\maketitle

\section{Introduction}
Equidistributed statistics in combinatorics can be quite nontrivial and interesting. 
A remarkable example is MacMahon's theorem on equidistribution of the number of inversions and the major statistic on permutations \cite{mm}, which has a conceptual bijective proof due to Foata \cite{foata}, see also \cite[Prop.~1.4.6]{ec1}. 

In this paper, we study an equidistribution result of a similar kind but for statistics on plane partitions, which is a rather unusual instance compared to permutations. 

%\vspace{0.5em}

\subsection{Plane partitions} 
A \textit{plane partition} is a matrix $\pi = (\pi_{i,j})_{i,j \ge 1}$ of nonnegative integers with finitely many nonzero entries such that
\begin{align*}
  \pi_{i,j} \ge \pi_{i+1, j}, \qquad \pi_{i,j} \ge \pi_{i, j+1}, \qquad \text{for all }i, j \ge 1.
\end{align*}
A plane partition can be identified with its diagram 
$$D(\pi) := \{(i,j,k): 1\le k \le \pi_{i,j}\},$$ 
which can be visually represented as a pile of 3d boxes, see Figure~\ref{fig:corner-example}. We define the set of {\it corners} of $\pi$ as follows: 
$$\mathrm{Cor}(\pi) := \left\{(i,j,k) \in D(\pi) : (i+1, j,k), (i,j+1,k) \not\in D(\pi) \right\}.$$ 
There are two natural statistics on plane partitions: 
$$
\text{the {\it volume} $|\pi| := \sum_{i,j} \pi_{i,j} = |D(\pi)|$ and the {\it trace} $\mathrm{tr}(\pi) := \sum_i \pi_{i, i}$.}
$$
There are also two other statistics on plane partitions introduced in \cite{yel21a, yel21b}: 
$$
\text{the {\it corner-hook volume} $|\pi |_{ch} := \sum_{(i,j,k) \in \mathrm{Cor}(\pi)} (i+j - 1)$ and %the {\it number of corners} 
$\mathrm{cor}(\pi) := |\mathrm{Cor}(\pi)|$. } 
$$

The set of plane partitions whose diagram lies inside the box $[a]\times [b]\times [c]$ is denoted by $\mathrm{PP}(a,b,c)$, where we use the notation $[n]:= \{1, \ldots, n \}$. 
%$For such plane partitions 

We prove that the bivariate statistics  
(\textit{volume}, \textit{trace}) and (\textit{corner-hook volume}, \textit{number of corners}) are jointly equidistributed over boxed plane partitions.

\begin{theorem}[Equidistribution of (\textit{volume}, \textit{trace}) and (\textit{corner-hook volume}, \textit{corners})]\label{th:main-equi}
 We have: 
  %For any integer $a,b,c \ge 1$,
  $$
    \sum_{\pi \in \mathrm{PP}(a,b,c)} q^{|\pi|} t^{\mathrm{tr}(\pi)}
    =
    \sum_{\pi \in \mathrm{PP}(a,b,c)} q^{|\pi|_{ch}} t^{\cor(\pi)}.
  $$
Equivalently, for all $n,k$ we have:
$$
|\{\pi \in \mathrm{PP}(a,b,c) : |\pi| = n, \mathrm{tr}(\pi) = k \}| = |\{\pi \in \mathrm{PP}(a,b,c) : |\pi|_{ch} = n, \mathrm{cor}(\pi) = k \}|.
$$
\end{theorem}

%  $$
%    \sum_{\pi \in \mathrm{PP}(a,b,c)} q^{|\pi|} t^{|\pi|_{ch}}
%    =?
%    \sum_{\pi \in \mathrm{PP}(a,b,c)} q^{|\pi|_{ch}} t^{|\pi|}
%  $$

A weaker (as an infinite sum when $c \to \infty$) version of this theorem was established in \cite{yel21b} %(and generalized in higher dimensions in \cite{ay23}) 
by an indirect argument via maps to matrices. In this paper, we show that a stronger result holds for a finite boxed version, by a direct combinatorial argument. %, which is the main purpose of this paper.

\vspace{0.5em}

The following volume-trace generating function for plane partitions was obtained by Stanley \cite{sta73}:
$$
\sum_{\pi\in \mathrm{PP}(a,b, \infty)} q^{|\pi|} t^{\mathrm{tr}(\pi)} = \prod_{i=1}^{a}\prod_{j=1}^{b}(1 - t q^{i + j - 1})^{-1},%\frac{1}{}
$$
which generalizes the volume generating function (for $t = 1$) by MacMahon \cite{macmahon}. Note also that for $t = 1$ and a finite boxed version there is a nice product formula for the corresponding generating function 
  $$
    \sum_{\pi \in \mathrm{PP}(a,b,c)} q^{|\pi|} %t^{\mathrm{tr}(\pi)}
    = \prod_{i= 1}^a \prod_{j = 1}^b \prod_{k = 1}^c \frac{1 - q^{i + j + k -1}}{1 - q^{i+j+k - 2}}.
$$
%which no longer holds when we consider the bivariate volume-trace generating function. 
In the same series of works, MacMahon also conjectured volume generating functions for $d$-dimensional partitions which naturally generalize plane partitions in higher dimensions. Later, in \cite{abmm} his conjecture was shown to be incorrect %by computer calculations already in the case 
for all $d \ge 3$. 
In \cite{yel21a, yel21b} the second author introduced the corner-hook volume on plane partitions (named `up-hook volume' there up to a diagram rotation) and showed that it gives product formulas as in the volume generating function. 
We generalized this result to $d$-dimenisonal partitions in \cite{ay23}, and it turns out that the corner-hook volume is the correct statistic lying behind MacMahon's generating functions for general $d$. Theorem~\ref{th:main-equi} shows that  the two statistics are equidistributed for $d = 2$, which does not happen in higher dimensions. %why MacMahon made wrong assumption -- the statistics are equidistributed in $d = 2$.

\vspace{0.5em}

%\subsection{Schur and dual stable Grothendieck polynomials}
In this paper, we in fact show a direct combinatorial proof of a more general identity with weights relating two known families of symmetric functions (indexed by ordinary partitions): \textit{Schur polynomials} $s_{\lambda}$ and \textit{dual stable Grothedieck polynomials} $g_{\lambda}$ which can be viewed as $K$-theoretic extensions of Schur functions, see \textsection\,\ref{sg1}, \ref{sg2} for definitions and context. 
%By $\mathrm{P}(a,b)$ we denote the set of (ordinary) partitions whose Young diagram lies inside the box $[a] \times [b]$.   
%\begin{theorem}\label{th:main-gr}
%We have: 
%  \begin{align*}
%    \sum_{\lambda \in \mathrm{P}(\min(a,b), c) }s_{\lambda}(x_1, \ldots, x_{a}) s_{\lambda}(z_1, \ldots, {z}_{b})
%    =
%      \sum_{\lambda \in \mathrm{P}(b, c)} g_{\lambda}(x_1, \ldots, {x}_{a}; z_1,\ldots, {z}_{b}) 
%  \end{align*}
%\end{theorem}
It is obtained by using the Lindstr\"om--Gessel--Viennot lemma \cite{lin, lgv} and interpreting plane partitions as non-intersecting path systems in two different ways on the same picture, as in Figure~\ref{fig:picturebox}. 
We then apply it to double enumeration of plane partitions. %For our knowledge, known proofs use Known proofs, see for instance \cite{sta, sta2} uses celebrated RSK correspondence or its analogies. 

The proofs use ideas developed in \cite{ay22}.

\subsection{Ordinary partitions}
As a byproduct of our approach, we also introduce a new statistic $|\cdot|_c$ for ordinary partitions, which we call the \textit{cohook area} and show that it is equidistributed with the usual area $|\cdot|$ of partitions.

\begin{theorem}\label{th:main-partitions}
  We have:
  $$
    \sum_{\lambda \in \mathrm{P}(a,b)} q^{|\lambda|} t^{d(\lambda)} = \sum_{\lambda \in \mathrm{P}(a,b)} q^{|\lambda|_c} t^{\mathrm{cor}(\lambda)},
  $$
  where sum runs over all partitions $\lambda$ liyng inside a box $a \times b$, $d(\lambda)$ is the length of the Durfee square of $\lambda$, and $\mathrm{cor}(\lambda)$ is the number of corners of $\lambda$.
\end{theorem}

\section{Preliminaries}
\subsection{Partitions}
A {\it partition} is a sequence $\lambda = (\lambda_1,\ldots,\lambda_\ell)$ of positive integers $\lambda_1 \ge \ldots \ge \lambda_\ell$, where $\ell(\lambda) = \ell$ is the {\it length} of $\lambda$. Denote $|\lambda| = \sum_{i} \lambda_i$ the {\it size} or {\it area} of $\lambda$.
Every partition $\lambda$ can be represented as the {\it Young diagram} $D(\lambda) := \{(i, j): i \in [1, \ell], j \in [1,\lambda_i], (i, j) \in \mathbb{N}^2 \}$.
By $\lambda'$ we denote the \textit{conjugate partition} of $\lambda$, i.e. transpose of its diagram. For a cell $(i,j) \in D(\lambda)$ the values $\lambda_i - i + 1$ and $\lambda'_j - j + 1$ are called the {\it arm} and {\it leg} lengths of $(i,j)$. Together, they constitute $(i,j)$ hook length $h_{\lambda}(i,j) := (\lambda_i - i) + (\lambda'_j - j) + 1$.

A \textit{Durfee square side} $d(\lambda)$ of partition $\lambda$ is the largest $k$ for which $\lambda_k \ge k$. The Frobenius notation $(a_1,\ldots,a_k | b_1 \ldots, b_k)$ of a partition $\lambda$ with Durfee square side equal to $k:= d(\lambda)$ expresses the partition in terms of its hooks for $(i,i) \in D(\lambda)$, namely\footnote{We slightly deviate from the standard definition by including the box $(i,i)$ to the arm and the leg.} \textit{the arm} $a_i = |\{(i,i),\ldots,(i,\lambda_i)\}|$ and \textit{the leg} $b_i = |\{(i,i),\ldots,(\lambda'_i, i)\}|$. Clearly, $a_1 > \ldots > a_k \ge 1$ and $b_1 > \ldots > b_k \ge 1$ and $\sum_{i=1}^k (a_i + b_i - 1) = |\lambda|$ is the partition area.

The set of partitions whose Young diagram lies inside the box $[a] \times [b]$ is denoted by $\mathrm{P}(a,b)$.

\subsection{Plane partitions}
% A \textit{plane partition} is a matrix $\pi = (\pi_{i,j})_{i,j \ge 1}$ of nonnegative integers with finitely many nonzero entries such that
% \begin{align*}
%   \pi_{i,j} \ge \pi_{i+1, j}, \qquad \pi_{i,j} \ge \pi_{i, j+1}, \qquad \text{for all }i, j \ge 1.
% \end{align*}
% Every plane partition $\pi$ can be represented as the diagram 
% $$D(\pi) := \{(i,j,k): 1\le k \le \pi_{i,j}\},$$ 
% which can be visually represented as a pile of 3d boxes, see Figure~\ref{fig:corner-example}. 
We refer the reader to the introduction section for the definition of a plane partitions.
The {\it shape} of $\pi$ is given by $\mathrm{sh}(\pi) = \{(i,j): \pi_{i,j} > 0\}$. By the \textit{side shape} $\mathrm{sh}_1(\pi)$ we denote the partition given by its first row $(\pi_{1,i})$.

Recall, that the box or the element $(i,j,k) \in D(\pi)$ is called \textit{a corner}, if
$$ (i+1,j,k) \not\in D(\pi) \quad\text{and}\quad (i,j+1,k) \not \in D(\pi).$$
If additionally $(i,j,k+1) \not\in D(\pi)$ we call it \textit{top corner}.
The set of all corners is denoted by $\Cor(\pi)$ and its number $\cor(\pi) = |\Cor(\pi)|$.
% Alternatively, 
% $$\mathrm{Cor}(\pi) = \{(i, j, k): \max(\pi_{i+1,j},\pi_{i,j+1}) < k \le \pi_{i,j}\}.$$
% The \textit{volume} $|\pi|$ and the \textit{corner-hook volume} $|\pi|_{ch}$ of $\pi$ are defined as follows:
% $$
%   |\pi| = \sum_{i,j\ge 1} \pi_{i,j}, \qquad |\pi|_{ch} = \sum_{(i,j,k) \in \Cor(\pi)} (i + j - 1).
% $$
% The \textit{trace} of $\pi$ is $\mathrm{tr}(\pi) := \sum_{i \ge 1} \pi_{i,i}$, i.e. the sum of diagonal entries.

The set of plane partitions whose diagram lies inside the box $[a]\times [b]\times [c]$ is denoted by $\mathrm{PP}(a,b,c)$
We regard plane partition $\pi \in \mathrm{PP}(a,b,c)$ as a pile of boxes.

\vspace{0.5em}
For example, consider the plane partition $\pi$ of shape $\sh(\pi)=(3,3,2,1)$ and the side shape $\sh_1(\pi) = (3,3,1)$ given by
\begin{align}\label{eq:pp}
    \pi = \begin{pmatrix}
        3 &3 &1\\
        3 &2 &1\\
        3 &2 &0\\
        1 &0 &0
      \end{pmatrix} \in \mathrm{PP}(4, 3, 3),
\end{align} 
which is displayed in Figure~\ref{fig:corner-example}. It has $\mathrm{tr}(\pi) = 5$, $|\pi| = 19$ and $|\pi|_{ch} = 2 + 4 + 3 + 4 + 4 + 4 = 21$, where
\begin{align*}
  \Cor(\pi) &= \{ (1,2,3), (2,3,1), (3,1,3), (3,2,2), (3,2,1), (4,1,1) \},
  \qquad
  \cor(\pi) = 6.%,\\
  %|\pi|_{ch} &= 4 + (5 + 4 + 4 + 3) + (4 + 4 + 3 + 3) = 35.
\end{align*}

\begin{remark}
    The definitions of corners follow conventions in \cite{ay23}. It is equivalent to the notion of `descents' in \cite{yel21a, yel21b}. 
\end{remark}

\begin{figure}
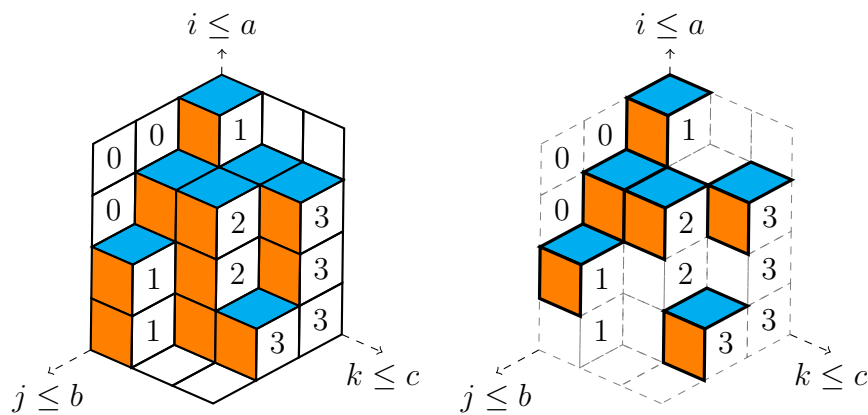
%[h]
    \centering
    % \picturebox
    \picturepp
    \picturecorner
    \caption{The plane partition $\pi$ from \eqref{eq:pp} represented as a pile of boxes. We present it in a rotated way, so that the side shape $\sh_1(\pi) = (3,3,1)$ appears at the bottom with rows corresponding to $y$-direction, columns to $z$-direction and height to $x$-direction (compared to Fig.~\ref{fig:picturebox}). The numbers on the right sides of boxes represent the heights $\pi_{i,j}$. }
  \label{fig:corner-example}
  \end{figure}

\subsection{Schur polynomials}\label{sg1}
We denote $\mathbf{x}_n = (x_1, \ldots, x_n)$ (and similarly for other sets of variables). 
The Schur polynomials $\{s_{\lambda}\}$ can be defined as follows:
$$
  s_{\lambda}(\mathbf{x}_n) := \sum_{\pi \in \mathrm{SPP}(\lambda)} \prod_{(i,j) \in D(\lambda)} x_{\pi_{i,j}}
$$
where $\mathrm{SPP}(\lambda)$ is the set of \textit{column-strict plane partitions} of shape $\lambda$, i.e. filling of a diagram $D(\pi)$ with entries $1,\ldots,n$ so that the rows are weakly decreasing and columns are strictly decreasing. They satisfy the following determinantal formula.

\begin{proposition}[N\"agelsbach–Kostka or dual Jacobi--Trudi identity]\label{prop:schur}
    For $n \ge 1$ and $\ell(\lambda) \le n$ we have:
    \begin{align}
        \label{eq:schur}
        s_{\lambda}(\mathbf{x}_n) = \det\left( e_{\lambda'_i - i + j}(\mathbf{x}_n) \right)_{i,j \in [\lambda_1]},
    \end{align}
    where $e_k(\mathbf{x}) = \sum_{i_1 < \ldots < i_k} x_{i_1} \cdots x_{i_k} = s_{(1^k)}(\mathbf{x})$ is the $k$-th elementary symmetric function. 
\end{proposition}

\subsection{Dual stable Grothendieck polynomials}\label{sg2}
The refined dual stable Grothendieck polynomials $\{g_{\lambda}(\mathbf{x}_a; \mathbf{z}_b)\}$ in two sets of variables $\mathbf{x}_a = (x_1,\ldots,x_a)$ and $\mathbf{z}_b = (z_1,\ldots,z_b)$, indexed by partitions $\lambda$, are defined as follows:
$$
  g_{\lambda}(\mathbf{x}_a; \mathbf{z}_b) := \sum_{\pi:~\sh_1(\pi)=\lambda}~\sum_{(i,j,k) \in \Cor(\pi)} x_{i} z_{j}
$$
where the sum runs over plane partitions $\pi \in \mathrm{PP}(a,b,c)$ of side shape $\sh_1(\pi) = \lambda \subseteq [b] \times [c]$. 

The dual stable Grothedieck polynomials $g_{\lambda}(\mathbf{x}) = g_{\lambda}(\mathbf{x}; \mathbf{z})|_{z_i \to 1}$ were introduced in \cite{lp} as a $K$-theoretic analogue of Schur functions. Their refined version $\tilde{g}(\mathbf{x}; \mathbf{z}) = \mathbf{z}^{\lambda}g(\mathbf{x}; \mathbf{z}^{-1})$ was introduced in \cite{ggl}. %and generalized version $\{\overline{g}_{\lambda}(\mathbf{x}, \mathbf{z}) = g_{\lambda}(\mathbf{z}, \mathbf{x})\}$ introduced in \cite{yel21b}.
They satisfy the following determinantal formula, which is equivalent to the one proved in \cite{yel17}. %\cite{ay22, kim22} for refined version with skew shape.
\begin{proposition}\label{prop:gr-det}
    %Let $a,b,c \ge 1$. 
    For any partition $\lambda$ with $\lambda_1 \le a$, we have:
    $$
        g_{\lambda}(\mathbf{x}_a; \mathbf{z}_b) = \det\left(
            z_{\lambda'_i}\sum_{k \ge 0} e_{j-i+k}(\mathbf{x}_a) e_{k}(\mathbf{z}_{\lambda'_i-1})
            \right)_{i,j \in [\lambda_1]}.
    $$
\end{proposition}

\section{Paths enumeration}%{Proof of theorem \ref{th:main-gr}}
\label{sec:main-th}

\subsection{Graph construction}\label{subsec:graph}
\newcommand{\coloredvector}[2]{%
  \tikz[baseline=0.5ex]{%
    \draw[->, thick, #2] (0,0) -- #1;
  }%
}
Let $a,b,c \in \mathbb{N}$ be fixed. 
Define a directed acyclic weighted graph $G$ %= (\mathbb{Z}^2 \cup \mathbb{Z}^2, E, w) 
with vertices on the lattice $\mathbb{Z}^3$ (in $X,Y,Z$ coordinates) as follows:
\begin{itemize}[leftmargin=20pt]
  \item The vertices are lattice points on the plane $XY$ ({\it wall}) and the plane $XZ$ ({\it floor}).
  \item The edges $\{e\}$ are of four types, for $i \in \mathbb{Z}$:
    \subitem $\left(\coloredvector{(0.5,0.5)}{blue}\right)$ floor {\it forward} edges: $(i,0,j) \to (i,0,j-1)$ of weight $w(e) = 1$, for $j \in [1, b]$;
    \subitem$\left(\ \ \coloredvector{(0,0.5)}{blue}\ \ \right)$ floor {\it left} edges: $(i,0,j) \to (i-1,0,j-1)$ of weight $w(e) = z_j$, for $j \in [1, b]$;
    \subitem $\left(\ \ \coloredvector{(0,0.5)}{black}\ \ \right)$ wall {\it upward} edges: $(i,j-1,0) \to (i,j,0)$ of weight $w(e) = 1$, for $j \in [1, a]$;
    \subitem $\left(\coloredvector{(0.5,0.5)}{black}\right)$ wall {\it right} edges: $(i-1,j-1,0) \to (i,j,0)$ of weight $w(e) = x_j$,  for $j \in [1, a]$.
\end{itemize}
The directions in the edge list are specified according to Figure~\ref{fig:picturebox} and  Figure~\ref{fig:graphexample}.
We also put $c$ sources $\mathbf{A} = (A_1,\ldots,A_c)$ and $c$ sinks $\mathbf{B} = (B_1,\ldots,B_c)$ for $i \in [c]$:
$$
  A_i = (i, 0, b), \qquad B_i = (i, a, 0),
$$
i.e. each path $P: A_i \to B_j$ first `crawls' on the floor and then `climbs' up by the wall. Note also that $\#\{\text{left steps}\} - \#\{\text{right steps}\} = i - j$, with an equal number of left and right steps if $i = j$.
\begin{figure}
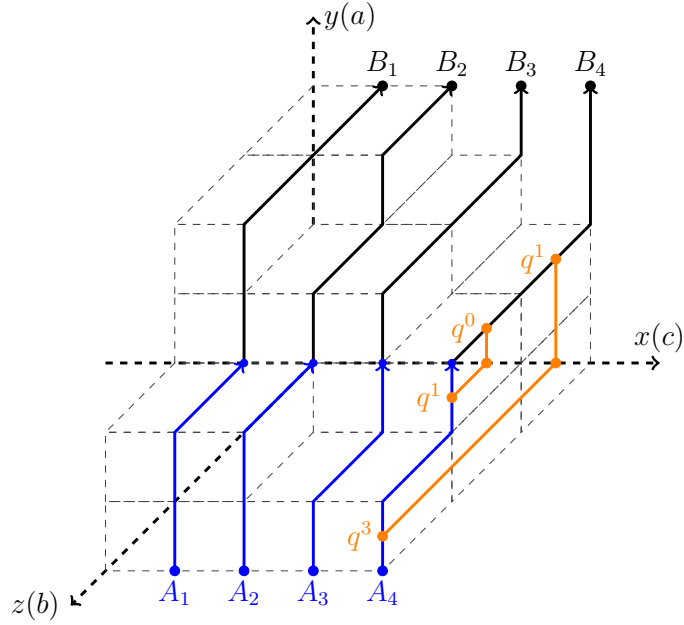
%[ht]
    \centering
    % \pictureboxmain
    \resizebox{0.589\textwidth}{!}{
        \pictureboxmainshift
    }
    \caption{This picture can be observed in three ways: 
    (a) boxed plane partition in $\mathrm{PP}(a,b,c)$, with dashed lines; (b) the paths $Q_i: A_i \to B_i$, where $i$-th one travels in the plane $x = i$, these paths define the dashed plane partition (see Theorem~\ref{th:main-cross}); 
    (c) the paths $P_i: A_i \to B_i$, each going first by the `floor' ($y = 0$ plane) and then by the `wall' ($z = 0$ plane). The orange lines represent the volume enumeration of (a) with edges of (b).}
    \label{fig:picturebox}
  \end{figure}
\begin{figure}
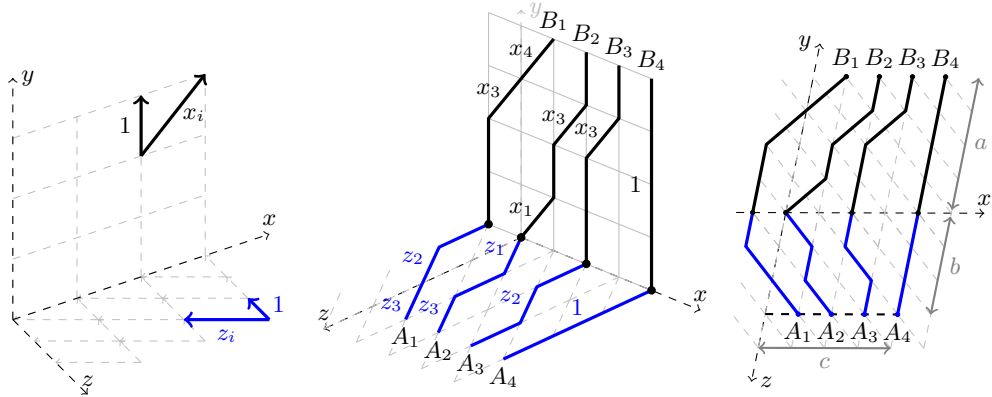
%[h]
    \centering
    {\scriptsize
    \picturegraph
    \picturenpath
    \picturenpathy    
    }
    \caption{(a) A part of the lattice with typical sample steps. (b) An example of a path system for $a = 4$, $b = 3$, $c = 4$ with the path weights given by $w(P_1) = z_3 z_2 x_3 x_4$, $w(P_2) = z_3 z_1 x_1 x_3$, $w(P_3) = z_2 x_3$, $w(P_4) = 1$. 
    (c) A planar view of the graph $G$. }
    \label{fig:graphexample}
\end{figure}

\subsection{Enumerators}
Define weighted path enumerator in the usual way:
$$
    w(X \to Y) = \sum_{P: X \to Y} \prod_{e \in P} w(e)
$$
over all paths $P$ in the lattice $G$ from the point $X \in G$ to the point $Y \in G$ with steps $e$ given as above. The weight of a path is a product of all edge weights.
The following formula is then clear from the construction.

\begin{lemma}
    We have: 
    $$w(A_i \to B_j) = \sum_{k \ge 0} e_{k}(\mathbf{z}_b) e_{j - i + k}(\mathbf{x}_a).$$
\end{lemma}
\begin{proof}
    To reach $B_j$ starting from $A_i$ via the path $P$, the path latter needs to have the shift $j - i$ by the coordinate $x$. If the path $P$ has $k \ge 0$ extra left steps, then the shift becomes $j - i + k$. Let $P \cap \{x\text{-axis}\} = C = (i, 0, i - k)$. Then
    $$
        w(A_i \to C) = e_{k}(\mathbf{z}_b), \qquad w(C \to B_j) = e_{j - i + k}(\mathbf{x}_a).
    $$
    Indeed, part of $P$ from $A_i$ to $C$ requires a choice of $k$ left steps out of $a$ possible, similarly from $C$ to $B_j$ requires $j-i+k$ right steps out of $b$ available.
\end{proof}

Similarly, define the signed weighted multi-enumerators
$$
    w(\mathbf{A} \to \mathbf{B}) := \sum_{\mathbf{P}: N^\pm(\mathbf{A}, \mathbf{B})} \sgn(\mathbf{P}) w(\mathbf{P})
$$
where $N^{\pm}(\mathbf{A}, \mathbf{B})$ is the set of all path systems $\mathbf{P} = (P_1, \ldots, P_c)$ from $\mathbf{A}$ to $\mathbf{B}$, such that $P_i: A_i \to B_{\sigma(i)}$ for some permutation $\sigma \in S_c$ (the twist) and $\sgn(\mathbf{P}) = \sgn(\sigma)$. 
Also let $N(\mathbf{A}, \mathbf{B})$ be the subset of nonintersecting path systems (i.e. $c$ paths with
no common vertices as in Fig.~\ref{fig:graphexample}(b)).
Set
$$
    D(a, b, c) := \det\left( \sum_{k \ge 0} e_{k}(\mathbf{z}_b) e_{j-i+k}(\mathbf{x}_a) \right)_{i,j \in [c]} = w(\mathbf{A} \to \mathbf{B}).
$$
\begin{lemma}%[LGV lemma]
\label{lemma:lgv}
    We have:
    \begin{align}\label{eq:det}
        D(a,b,c) = \sum_{\mathbf{P} \in N(\mathbf{A}, \mathbf{B})} w(\mathbf{P})
    \end{align}
    where $\mathbf{P} = (P_1,\ldots,P_c)$ with $P_i: A_i \to B_i$.
\end{lemma}
\begin{proof}
    The proof follows by applying the Lindstr\"om--Gessel--Viennot lemma \cite{lin, lgv} for the graph $G$ (defined in Subsection \ref{subsec:graph}) with sources $\mathbf{A}$ and sinks $\mathbf{B}$. Since the graph is in fact planar, the twist $\sigma = \mathrm{id}$ for each $\mathbf{P} \in N(\mathbf{A}, \mathbf{B})$.
\end{proof}

%We are now ready to prove main theorem.

\begin{lemma}\label{lm:main-gr}
%\begin{theorem}[= Theorem~\ref{th:main-gr}]
 We have: 
  %Let $a,b,c\in\mathbb{N}$. Then
  \begin{align*}
    \sum_{\lambda \in \mathrm{P}(\min(a,b), c)}s_{\lambda}(\mathbf{x}_{a})s_{\lambda}(\mathbf{z}_{b})
    =
      \sum_{\mu \in \mathrm{P}(b, c)} g_{\mu}(\mathbf{x}_{a}; \mathbf{z}_{b}) 
  \end{align*}
%\end{theorem}
\end{lemma}

\begin{proof}
    By Lemma~\ref{lemma:lgv}, non-intersecting path systems from $\mathbf{A}$ to $\mathbf{B}$ are enumerated by the determinant $D(a, b, c)$. Let us enumerate these path systems in two ways. %, interpreting both sides as non-intersecting paths systems to Figure~\ref{fig:picturebox}.
    
    On the one hand, let $\mathbf{P} \in N(\mathbf{A}, \mathbf{B})$ be a non-intersecting path system. Let $\mathbf{C}=(C_1,\ldots,C_c)$ be intersection points of paths $P_i$ with $OX$ line for $i \in [c]$. Then $\mathbf{C}$ split the path system $\mathbf{P}$ into two path systems $\mathbf{P}_z: \mathbf{A} \to \mathbf{C}$ (floor) and $\mathbf{P}_x: \mathbf{C} \to \mathbf{B}$ (wall). Let $C_i = (i-\lambda'_i, 0, 0)$ for some vector $\lambda' = (\lambda'_1,\ldots\lambda'_c)$. Since $\mathbf{P} \in N(\mathbf{A},\mathbf{B})$, the vector $\lambda'$ is a partition. Iterating over $\mathbf{C}$, by Proposition~\ref{prop:schur} we may decompose the determinant as follows: 
    \begin{align*}
        D(a, b, c) 
        = 
        &\sum_{\mathbf{C}} \det\left(w(A_i \to C_j)\right)
        \det\left(w(C_i \to B_j)\right)
        =\\
        &\sum_{\lambda} \det\left(
            e_{\lambda'_i + j - i}(\mathbf{z}_b)
            \right)
            \det\left(
                e_{\lambda'_i + j - i}(\mathbf{x}_a)
                \right)
                = \sum_{\lambda} s_{\lambda}(\mathbf{z}_b) s_{\lambda}(\mathbf{x}_a),
    \end{align*}
    where the sum runs over partitions $\lambda$ with $D(\lambda) \subseteq [\min(a,b)] \times [c]$. Indeed, $\lambda_1 \le c$ since there are $c$ paths corresponding to columns of $D(\lambda)$ diagram, and $\ell(\lambda) \le \min(a,b)$ since the terms $s_{\lambda}(\mathbf{x}_a)s_{\lambda}(\mathbf{z}_b)$ vanish otherwise.

    On the other hand, for each $i \in [c]$, let $\mu'_i$ be the maximal index for which the path $P_i$ has $z_{\mu'_i}$ in its weight $w(P_i)$; in other words, index of the first left edge of $P_i$. If there is no such edge, we set $\mu'_i = 0$. Since the paths are non-intersecting, $\mu' = (\mu'_1, \ldots, \mu'_c)$ is a partition satisfying $D(\mu) \subseteq [b] \times [c]$. Indeed, it has at most $c$ columns (as there are $c$ paths) and each $\mu'_i \le b$ by construction. 

    Denote by $A'_i = (i-1, 0, \mu'_i-1)$ the point of $P_i$ after passing this edge for $i \in [\mu_1]$. Summing up over all possible $\mathbf{A'} = (A'_1, \ldots, A'_c)$ (or equivalently, over $\mu'$) and using Proposition~\ref{prop:gr-det} we get:
    \begin{align*}
        D(a, b, c) 
        = &\sum_{\mathbf{A}'} \det\left(
            z_{\mu'_i} w(A'_i \to B_j)
            \right)\\
            = &\sum_{\mu'} \det\left(
                z_{\mu'_i} \sum_{k \ge 0} e_k(\mathbf{z}_{\mu'_i-1}) e_{j - i + k}(\mathbf{x_a}) 
                \right)
                = \sum_{\mu'} g_{\mu}(\mathbf{x}_a; \mathbf{z}_b).
    \end{align*}
    This completes the proof.
\end{proof}

\begin{remark}
This identity was shown in \cite{ms} by a probabilistic argument using the connection of dual stable Grothendieck polynomials with last passage percolation model from \cite{yel-lpp} and the Schur measure.
\end{remark}

\begin{remark}
    %Denote $\mathbb{N}(a, b)$ to be set of $\mathbb{N}$-matrices of size $a \times b$. 
    We now have a complete `triangle' of formulas and relations between them (for $c \to \infty$):
    \begin{itemize}
        \item The RSK correspondence (see e.g. \cite[Ch.~7]{sta}) shows  
        $$\sum_{\ell(\lambda) \le \min(a,b)}s_{\lambda}(\mathbf{x}_a) s_{\lambda}(\mathbf{z}_b)
        = \prod_{i=1}^a \prod_{j=1}^b \frac{1}{1 - x_i z_j}$$
        
        \item The bijective map $\Phi: \mathrm{PP}(a,b,\infty) \to \{(a_{i,j})_{i \in [a], j \in [b]} : a_{i,j} \in \mathbb{N}\}$ 
        (see \cite{yel21a, yel21b, ay23}) shows
        $$\sum_{\ell(\lambda) \le b} g_{\lambda}(\mathbf{x}_a; \mathbf{z}_b)
        = \prod_{i=1}^a \prod_{j=1}^b\frac{1}{1 - x_i z_j}
        $$
        \item Lemma~\ref{lm:main-gr} directly shows %(for $c \to \infty$):
        $$\sum_{\ell(\lambda) \le \min(a,b)}s_{\lambda}(\mathbf{x}_a) s_{\lambda}(\mathbf{z}_b)
        = \sum_{\ell(\lambda) \le b} g_{\lambda}(\mathbf{x}_a;\mathbf{z}_b).
        $$
    \end{itemize}
\end{remark}

\begin{remark}\label{remark:direct}
  Our proof can be converted to a direct bijection using `jeu de taquin' like operations on paths described in \cite{ay22}. 
  \begin{figure}[h]
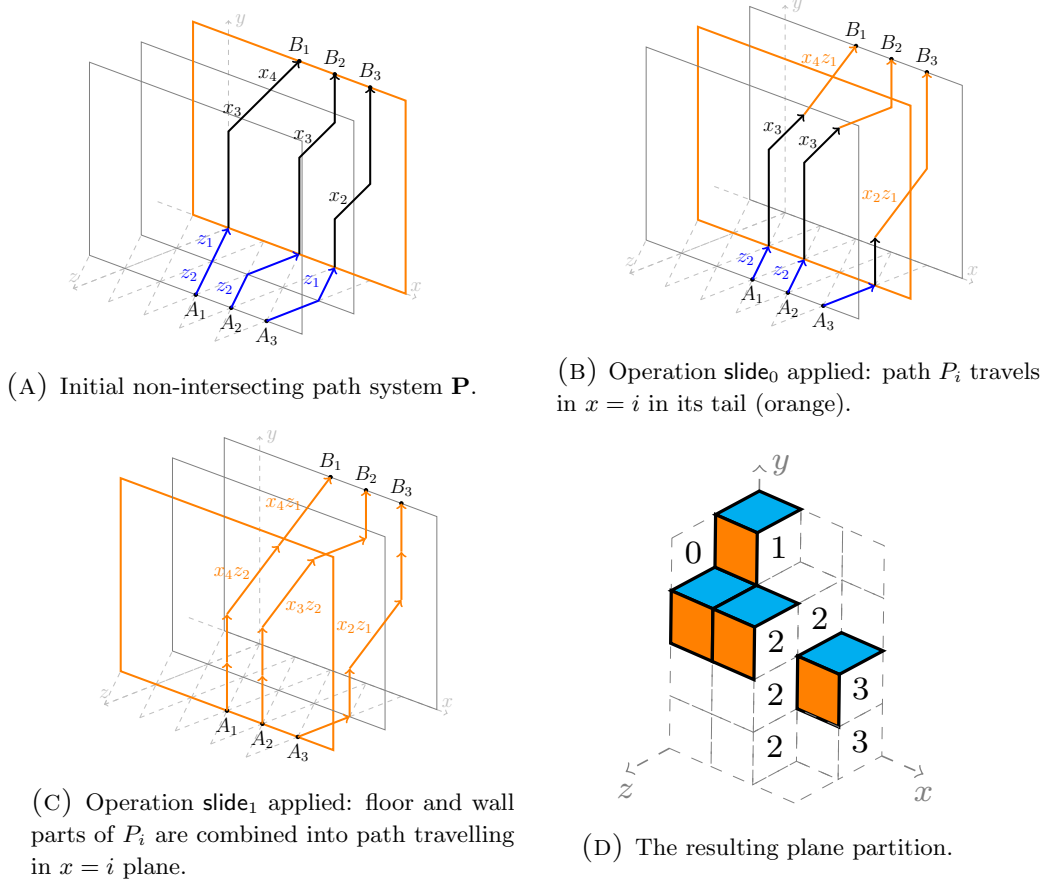

  \begin{center}
      \centering
      \begin{minipage}{0.4\textwidth}
          \centering
          \resizebox{0.8\textwidth}{!}{
              \pictureslideA
          }
          \subcaption{{\scriptsize Initial non-intersecting path system $\mathbf{P}$.}}
      \end{minipage}
      \qquad 
      %\quad
      \begin{minipage}{0.4\textwidth}
          \centering
          \resizebox{0.8\textwidth}{!}{
              \pictureslideB
      }
      \subcaption{\scriptsize Operation $\mathsf{slide}_0$ applied: path $P_i$ travels in $x = i$ in its tail (orange).
      }
      \end{minipage}
      \qquad
      \begin{minipage}{0.4\textwidth}
          \centering
          \resizebox{0.8\textwidth}{!}{
              \pictureslideC
          }
          \subcaption{\scriptsize Operation $\mathsf{slide}_1$ applied: floor and wall parts of $P_i$ are combined into path travelling in $x = i$ plane.
          }
      \end{minipage}
      %\quad
      \begin{minipage}{0.4\textwidth}
          \centering
          \resizebox{0.8\textwidth}{!}{
              \pictureslideD
          }
          \subcaption{\scriptsize The resulting plane partition. }
      \end{minipage}
  \caption{From a non-intersecting path system to plane partition. }
  \label{fig:slide}
  \end{center}
\end{figure}
  In particular, we defined the operations $\{\mathrm{slide}_k\}_{k \ge 0}$,\footnote{As their definitions are somewhat technical and long, we do not reproduce them here. } %Operation $\mathrm{slide}_k$ 
  describing procedure which transforms 3d non-intersecting path system $\mathbf{P}^{(i)}: \mathbf{A} \to \mathbf{B}$ to 3d non-intersecting path system $\mathbf{P}^{(i+1)}: \mathbf{A} \to \mathbf{B}$. 
  Each path $P^{(i)}_j$ travels as follows: first crawls by the floor $y = 0$, then climbs up the wall $z = k$, and then travels on the plane $x = i$. 
  Then starting from $\mathbf{P}^{(0)} := \mathbf{P} \in N(\mathbf{A}, \mathbf{B})$ and applying consequently $\mathrm{slide}_{b-1}\circ\ldots\circ \mathrm{slide}_{0}$ to $\mathbf{P}$ we can construct a new path system $\mathbf{P}': \mathbf{A}\to \mathbf{B}$ with each $P'_i$ travelling in $x = i$ plane, with steps $(0,1,0)$, $(0,0,1)$ and $(0,1,1)$, where the latter step type corresponds to corners of the new plane partition, see Figure~\ref{fig:slide} (cf. \cite[Fig.~7]{ay22}). This defines a direct weight preserving bijection between plane partitions enumerated by corners in $g_{\lambda}(\mathbf{x}; \mathbf{z})$ and non-intersecting path systems in $D(a,b,c)$.
\end{remark}

\section{Equidistribution theorem}
In this section we show implications of Section~\ref{sec:main-th}. %are demonstrated. 

\begin{theorem}[= Theorem~\ref{th:main-equi}]\label{th:sec-equi}\label{th:main-cross}
  We have
  \begin{align}\label{eq:th-equi}
    \sum_{\pi\in \mathrm{PP}(a,b,c)} q^{|\pi|}t^{\mathrm{tr}(\pi)} = \sum_{\pi \in \mathrm{PP}(a,b,c)} q^{|\pi|_{ch}}t^{\mathrm{cor}(\pi)}
  \end{align}
  % where $\pi$ runs over $\mathrm{PP}(a,b,c)$.
\end{theorem}
\begin{proof}
  We refer to the Figure~\ref{fig:picturebox} for visualization. 
  In Lemma~\ref{lm:main-gr} set $x_i = q^{i}, z_i = tq^{i-1}$ into $D(a,b,c)$ to obtain:
  \begin{align*}
    \sum_{\lambda \subseteq [\min(a,b)]\times [c]} s_{\lambda}(q,q^2,\ldots,q^{a}) s_{\lambda}(t,tq,\ldots,tq^{b-1}) 
    &= \left[ \sum_{P \in N(\mathbf{A}, \mathbf{B})} w(\mathbf{P}) \right]_{x_i \mapsto q_i,\ z_i \mapsto tq^{i-1}}
    \\
    &=
    \sum_{\lambda \subseteq [b] \times [c]} g_{\lambda}(q,\ldots,q^{a}; t,\ldots,tq^{b-1}).
\end{align*}
  On the RHS we already obtain enumeration of plane partitions by corner-hook volume:
  \begin{align*}
    g_{\lambda}(q,\ldots,q^{a}; t,\ldots,tq^{b-1})
    = &\sum_{\pi \in \mathrm{PP}(a,b,c):~\mathrm{sh}_1(\pi) = \lambda} ~\prod_{(i,j,k) \in \mathrm{Cor}(\pi)} q^{j} \cdot t q^{i-1} 
    \\= &\sum_{\pi \in \mathrm{PP}(a,b,c):~\mathrm{sh}_1(\pi) = \lambda} ~t^{\cor(\pi)} q^{|\pi|_{ch}}.
  \end{align*}
  Hence, the RHS enumerates plane partitions by corners in the box $\mathrm{PP}(a,b,c)$.

  On the LHS, we interpret each term as a non-intersecting path system $\mathbf{P}$. View each path $P_i \in \mathbf{P}$ in other way. Project path $P_i: A_i \to B_i$ to the plane $x = i$ along the vector $(1,1,1)$ to obtain the path $Q_i: A_i \to B_i$ which uses $(0,1,0)$ and $(0,0,1)$ steps and travels within the plane $x = i$. It is easy to see (visually) that this is a bijection. See Figure~\ref{fig:projection} for a visual explanation.
  
  The path $Q_i$ is then the boundary of some partition $\lambda^{(i)} \subseteq [a] \times [b]$ (of its diagram), drawn in the plane $x = i$ (in French notation). 
  Since the paths are non-intersecting, $\pi := (\lambda^{(1)}\supseteq\ldots \supseteq \lambda^{(c)})$ is a plane partition, where $\pi_i := \lambda^{(i)}$ is its $i$-th row. This defines a bijection between non-intersecting path systems $\mathbf{P}$ enumerated by $D(a,b,c)$ and plane partitions $\pi \in \mathrm{PP}(a,b,c)$.

\begin{figure}
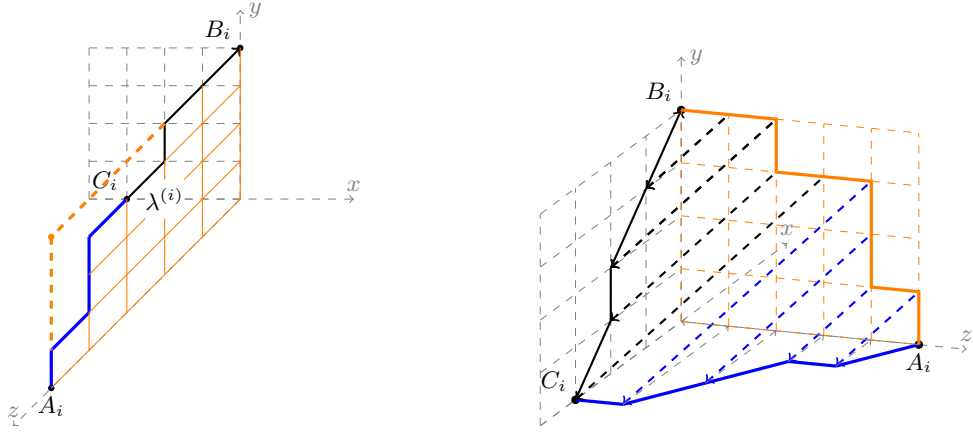
%[h]
    \centering
    {\scriptsize
    \pictureprojA
    \qquad\qquad\qquad \pictureprojB
    }
    \caption{The path $P_i$ from $A_i = (i,0,c)$ to $B_i = (i,b,0)$ intersecting OX at point $C_i = (i-\ell_i, 0, 0)$. (a) The Path $P_i$ and $Q_i$ coincide under this view angle.
    (b) The path $Q_i$ (orange) travels on the plane $x = i$ and bounds the partition $\lambda^{(i)} = (5,4,4,2)$, the projection lines are displayed to see the correspondence between the paths $P_i$ and $Q_i$.
    }
    \label{fig:projection}
\end{figure}

  Let us show that the contribution of $w(\mathbf{P})$ matches the required term $t^{\mathrm{tr}(\pi)} q^{|\pi|}$.
  For each $i \in [c]$ write the partition $\lambda^{(i)} = (a_1,\ldots,a_{\ell_i} | b_1 \ldots, b_{\ell_i})$ in Frobenius notation, where $\ell_i := d(\lambda^{(i)})$. Then the path $P_i$ has $2\ell_i$ edges, and moreover, the hook $(a_j, b_j)$ corresponds to the pair of edges of $P_i$: floor left edge $z_{b_j}$ and wall right edge $x_{a_j}$ (see orange lines in Figure~\ref{fig:picturebox}). Thus, for each $i \in [c]$ we have 
  $$
    w(P_i) = \prod_{j=1}^{\ell_i} z_{b_{j}} x_{a_j} 
    = \prod_{j=1}^{\ell_i} q^{b_j} \cdot t q^{a_j-1} 
    = q^{\sum (a_j + b_j - 1)} t^{\ell_i} = q^{|\lambda^{(i)}|} t^{\ell_i}
  $$
  and consequently
  $$
    w(\mathbf{P}) = \prod_{i=1}^a w(P_i) = q^{\sum |\lambda^{(i)}|}~ t^{\sum \ell_i} = q^{|\pi|} t^{\mathrm{tr}(\pi)},
  $$
  since $\mathrm{tr}(\pi) = \sum_{i=1}^c d(\lambda^{(i)})$.
  This completes the proof.
\end{proof}

As a byproduct of the above discussions we obtain the following corollaries.
\begin{corollary}[Fixed trace generating function]
  For %any $a,b,c \in \mathbb{N}$ and
  a partition $\lambda \in  \mathrm{P}(\min(a,b), c)$ we have:
  $$
  s_{\lambda}(q,q^2,\ldots,q^{a}) s_{\lambda}(t,tq,\ldots,tq^{b-1}) = \sum_{\pi \in \mathrm{PP}(a,b,c): ~\mathrm{Tr}(\pi) = \lambda} q^{|\pi|} t^{\mathrm{tr}(\pi)},
  $$
  where $\mathrm{Tr}(\pi) := (\pi_{1,1},\pi_{2,2}, \ldots, \pi_{d(\pi),d(\pi)})$ is the vector of non-zero diagonal entries.
\end{corollary}

\begin{corollary}[Boxed PP volume enumeration]
  For %any $a,b,c \in \mathbb{N}$ and
  a partition $\lambda \in  \mathrm{P}(\min(a,b), c)$ we have:
  $$
  D(a,b,c) |_{x_i \to q^i, z_j \to q^{j-1}} = \sum_{\pi \in \mathrm{PP}(a,b,c)} q^{|\pi|}.
  $$
\end{corollary}

\begin{remark}
Notably, the latter corollary is obtained without the use of the RSK correspondence. Indeed, $D(a,b,c)$ enumerates non-intersecting paths by LGV lemma (Lemma~\ref{lemma:lgv}) and the interpretation of $w(\mathbf{P})$ for $\mathbf{P} \in N(\mathbf{A}, \mathbf{B})$ after specialization is direct and visual.
\end{remark}

\begin{remark}
Since the volume and the corner-hook volume are equidistributed, it is natural to ask if they also have {\it symmetric joint distribution}, i.e. 
whether 
$$
\sum_{\pi} q^{|\pi|} t^{|\pi|_{ch}} = \sum_{\pi} q^{|\pi|_{ch}} t^{|\pi|_{}}
$$
holds. While (as we checked) this equality does {\it not} hold over all boxed plane partitions, it can be proved that it holds for sums over $\pi \in \mathrm{PP}(2,2,c)$. 
\end{remark}

\section{Partition corner hook area}
In this section we recover $1$-dimensonal version of the equidistriution theorem, see Remark~\ref{remark:path} for path interpretation.

For a partition $\lambda$, recall that $|\lambda| = |D(\lambda)|$ is its area, $d(\lambda) := \max_{\lambda_k \ge k}  k$ is the size of its Durfee square, and $h_\lambda(i,j) := (\lambda_i - i) + (\lambda'_j - j) + 1$ denotes the \textit{hook length} of the cell $(i,j)$. We clearly have
$$
  |\lambda| = \sum_{i=1}^{d(\lambda)} h_{\lambda}(i, i).
$$
Denote by $p(n, k)$ the number of partitions $\lambda$ with $|\lambda|=n$ and $d(\lambda)=k$ and set $p(n) = \sum_k p(n,k)$ the total number of partitions of $n$.

\begin{definition}
  Let $\lambda$ be a partition. The cell $(i, j) \in D(\lambda)$ is called a \textit{corner}, if $(i+1,j)$ and $(i,j+1)$ are not in $D(\lambda)$. The set of corners is denoted by $\mathrm{Cor}(\lambda)$ and its size is $\mathrm{cor}(\lambda) := |\mathrm{Cor}(\lambda)|$. 
  By $\mathrm{ch}_{}(i,j) := i + j - 1$ we denote a \textit{cohook length} of the cell $(i,j)$. Define the \textit{cohook area} $|\lambda|_{c}$ as follows: 
  $$
    |\lambda|_{c} := \sum_{(i,j)\in\mathrm{Cor}(\lambda)} \mathrm{ch}(i, j).
  $$
  For example, the partition $\lambda = (2,1)$ with $|\lambda| = 3$ has two corners $(1,2)$ and $(2,1)$, and hence,  $|\lambda|_{c} = 2 + 2 = 4$. 
  Denote by $q(n, k)$ the number of partitions with $|\lambda|_c = n$ and $\mathrm{cor}(\lambda)=k$ and set $q(n) = \sum_{k} q(n, k)$.
  
  Since any partition $\lambda$ is uniquely determined by its corners, let us introduce the \textit{corner notation}. By $[a_1,\ldots,a_k | b_1,\ldots, b_k]$ we denote a partition with corners at $(a_1,b_1),\ldots,(a_k,b_k)$, so that  
  $$
    a_1 > \ldots > a_k \ge 1 \quad\text{ and }\quad 
    1 \le b_1 < \ldots < b_k,
  $$
  since no two corners can occupy the same row or column. For example, 
  \begin{align*}
      \lambda = [4,2,1|3,5,6] &= \ytableausetup{textmode, boxsize=2em}\ytableaushort
                {{*(orange)} \none \none \none \none {{\small (1,6)}},
                 \none {*(orange)} \none \none {{\small (2,5)}},
                 \none \none {*(orange)},
                 \none \none {{\small (4,3) }}
                } * {6,5,3,3} = (6,4,1|4,3,2),\\
    &\text{with }|\lambda|_c = 18\text{ and }|\lambda| = 17.
  \end{align*}
\end{definition}

\begin{theorem}
  For any $n, k \ge 1$ we have:
  $$
    p(n, k) = q(n, k),
  $$
  i.e. the paired partition statistics $(|\cdot|, d(\cdot))$ and $(|\cdot|_c, \mathrm{cor}(\cdot))$ are equidistributed.
\end{theorem}
\begin{proof}
  Let $\lambda = (a_1,\ldots,a_k | b_1, \ldots, b_k)$ be a partition of $n$ written in Frobenious notation. Consider the following map 
  $$
  \Phi: \lambda = (a_1,\ldots,a_k | b_1, \ldots, b_k) \mapsto [b_1,\ldots,b_k | a_k, \ldots, a_1] =: \mu
  $$
  from Frobenius to corner notation. 
  Then we have $|\lambda| = \sum (a_i + b_i - 1) = \sum (a_i + b_{k+1-i} - 1) = |\mu|_{c}$.
  This shows that $\Phi$ is in fact a bijection 
  $$
    \Phi: \{\lambda: |\lambda| = n,\ d(\lambda)=k\} \to \{ \lambda: |\lambda|_{c} = n,\ \mathrm{cor}(\lambda)=k \}
  $$ 
  which implies the claim.
\end{proof}
\begin{example}
Let us illustrate the bijection on the partition $\lambda=(6,4,1|4,3,2)$, the result is equal to $\Phi(\lambda) = \mu = [2,3,4|6,4,1]$:
\newcommand{\yc}[4]{%
  \begin{scope}[shift={({#1},{#2})}]
    \draw[black] (0,0) rectangle (1,1);
    \begin{scope}
      \clip (0,0) rectangle (1,1);
      #3
    \end{scope}
    \node at (0.5,0.5) {#4};
  \end{scope}
}
% pattern shorthands
\newcommand{\porange}{%
  \fill[pattern={Lines[angle=45,distance=3pt,line width=0.35pt]},
        pattern color=orange,
        opacity=0.6] (0,0) rectangle (1,1);}

\newcommand{\pblue}{%
  \fill[pattern={Lines[angle=-45,distance=3pt,line width=0.35pt]},
        pattern color=cyan!70!blue,
        opacity=0.6] (0,0) rectangle (1,1);}

\newcommand{\pgray}{%
  \fill[pattern={Lines[angle=0,distance=3pt,line width=0.35pt]},
        pattern color=gray,
        opacity=0.6] (0,0) rectangle (1,1);}
\[
\Phi:\qquad
\lambda=
\begin{tikzpicture}[baseline={(current bounding box.center)},scale=0.8]
  % row 1
  \yc{0}{3}{\porange}{\scriptsize$(4,6)$}
  \yc{1}{3}{\porange}{}
  \yc{2}{3}{\porange}{}
  \yc{3}{3}{\porange}{}
  \yc{4}{3}{\porange}{}
  \yc{5}{3}{\porange}{}
  % row 2
  \yc{0}{2}{\porange}{}
  \yc{1}{2}{\pblue}{\scriptsize$(3,4)$}
  \yc{2}{2}{\pblue}{}
  \yc{3}{2}{\pblue}{}
  \yc{4}{2}{\pblue}{}
  % row 3
  \yc{0}{1}{\porange}{}
  \yc{1}{1}{\pblue}{}
  \yc{2}{1}{\pgray}{\scriptsize$(2,1)$}
  % row 4
  \yc{0}{0}{\porange}{}
  \yc{1}{0}{\pblue}{}
  \yc{2}{0}{\pgray}{}
\end{tikzpicture}
\qquad\longrightarrow\qquad
\mu=
\begin{tikzpicture}[baseline={(current bounding box.center)},scale=0.8]
  % row 1
  \yc{0}{3}{\porange}{}
  \yc{1}{3}{}{}
  \yc{2}{3}{}{}
  \yc{3}{3}{\pblue}{}
  \yc{4}{3}{}{}
  \yc{5}{3}{\pgray}{}
  % row 2
  \yc{0}{2}{\porange}{}
  \yc{1}{2}{\porange}{}
  \yc{2}{2}{\porange}{}
  \yc{3}{2}{\porange\pblue}{}
  \yc{4}{2}{\porange}{}
  \yc{5}{2}{\porange\pgray}{\scriptsize$(2,6)$}
  % row 3
  \yc{0}{1}{\pblue\porange}{}
  \yc{1}{1}{\pblue}{}
  \yc{2}{1}{\pblue}{}
  \yc{3}{1}{\pblue}{\scriptsize$(3,4)$}
  % row 4
  \yc{0}{0}{\porange}{\scriptsize$(4,1)$}
\end{tikzpicture}.
\]
\end{example}
\vspace{0.5em}
The following corollary is immediate.

\begin{corollary}
  For any $n \ge 1$, we have: $p(n) = q(n)$.
\end{corollary}

Analogous to Theorem~\ref{th:main-equi}, we present the boxed version.

\begin{theorem}[= Theorem~\ref{th:main-partitions}]\label{th:partequidistr}
  We have the following equidistribution of paired statistics $(|\cdot|, d(\cdot))$ and $(|\cdot|_{c}, \mathrm{cor}(\cdot))$ within the box $[m] \times [n]$:
    $$
        \sum_{\lambda \in \mathrm{P}(m, n)} q^{|\lambda|}t^{d(\lambda)} = \sum_{\lambda \in \mathrm{P}(m, n)} q^{|\lambda|_{c}} t^{\mathrm{cor}(\lambda)}.
    $$
\end{theorem}
\begin{proof}
  We are enough to verify that $D(\Phi(\lambda)) \subseteq [m] \times [n]$ whenever $D(\lambda) \subseteq [m] \times [n]$.
  Apply the map $\Phi$ to such $\lambda$ written in Frobenius notation as $(a_1,\ldots,a_k | b_1, \ldots, b_k)$. Then $\{a_i\}$ is a strictly decreasing subsequence of $[n]$, and similarly $\{b_i\}$ of $[m]$. Let $\mu := \Phi(\lambda) = [b_1,\ldots,b_k | a_k, \ldots, a_1]$ in corner notation, then each $a_i \in [n]$ and $b_i \in [m]$ ensuring that $D(\mu) \subseteq [m] \times [n]$. The reverse direction is similar.
\end{proof}

\begin{remark}\label{remark:path}
The bijection $\Phi$ can be (and in fact, was) recovered from Figure~\ref{fig:projection} or Figure~\ref{fig:phi}. Let $P$ be a path in $G$ from $A=(0,0,m)$ to $B=(0,n,0)$ whose projection along $(1,1,1)$ onto the plane $x=0$ traces the diagram of $\lambda$. Applying the slide operations of \cite{ay22} (see also Remark~\ref{remark:direct}) to $P$, we obtain a path $P':A\to B$ lying entirely on the plane $x=0$. Namely, 
define $P^{(0)}=P$, and for $k=1,\dots,m$ obtain $P^{(k)}$ from $P^{(k-1)}$ by moving its wall part from $z=k-1$ to $z=k$ along the corresponding floor edge, extending the $x=0$ part accordingly. %Then $P^{(k)}$ runs along the floor, climbs the shifted wall $z=k$, and continues in the plane $x=0$. 
Finally, set $P'=P^{(m)}$.
This may be viewed as continuous printing on the plane $x=0$, with the wall part as the printing head and the floor part as its trajectory.
\begin{figure}[h]
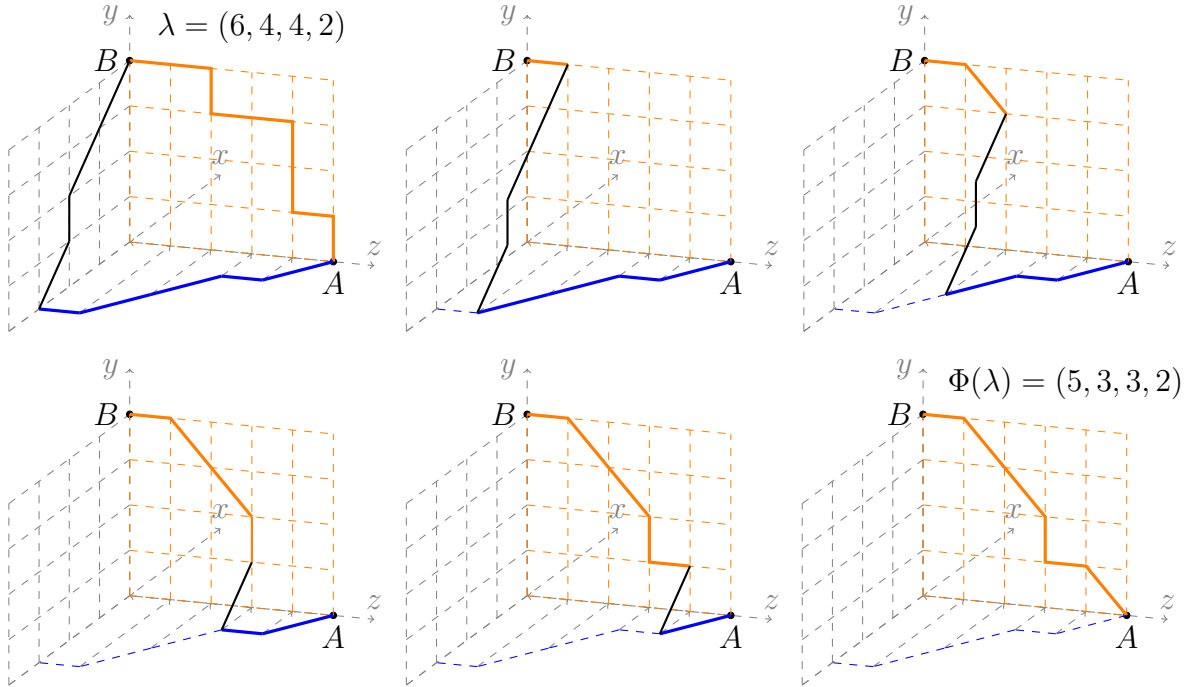

    \centering
    \begin{minipage}{0.32\textwidth}
        \centering
        \pictureoneslideA
    \end{minipage}
    \begin{minipage}{0.32\textwidth}
        \centering
        \pictureoneslideB
    \end{minipage}
    \begin{minipage}{0.32\textwidth}
        \centering
        \pictureoneslideC
    \end{minipage}

    \medskip

    \begin{minipage}{0.32\textwidth}
        \centering
        \pictureoneslideD
    \end{minipage}
    \begin{minipage}{0.32\textwidth}
        \centering
        \pictureoneslideE
    \end{minipage}
    \begin{minipage}{0.32\textwidth}
        \centering
        \pictureoneslideF
    \end{minipage}
    \caption{Path interpretation of the bijection $\Phi$.}\label{fig:phi}
\end{figure}
\end{remark}

%This establishes Theorem~\ref{th:main-partitions}. 

\section*{Acknowledgements}
We are grateful to Askar Dzhumadil'daev for encouraging us to write this paper and many inspiring conversations. 
We also thank the referees for useful comments. 

%This research was supported by the Science Committee of the Ministry of Science and Higher Education of the Republic of Kazakhstan (grant No. AP14869221).

\

%\newpage

\end{document}